\documentclass[english,reqno,10pt]{amsart}
\usepackage[T1]{fontenc}
\usepackage[latin9]{inputenc}
\usepackage[letterpaper, margin=4.1cm]{geometry}
\usepackage{amsthm}
\usepackage{amssymb}
\usepackage{setspace}
\usepackage{enumerate}
\usepackage{fancyhdr}
\usepackage{xcolor}
\usepackage{upgreek} 
\usepackage{mathabx}
\usepackage[colorlinks=true]{hyperref}
\hypersetup{linkcolor=red,citecolor=blue,filecolor=dullmagenta,urlcolor=blue}
\usepackage[shortlabels]{enumitem}
\usepackage{cancel}
\usepackage{graphicx}

\numberwithin{equation}{section}
\theoremstyle{plain}
\newtheorem{thm}{Theorem}[section]
\newtheorem{lem}[thm]{Lemma}
\newtheorem{prop}[thm]{Proposition}
\newtheorem{cor}[thm]{Corollary}
\newtheorem{claim}[thm]{Claim}
\newtheorem{rem}{Remark}[section]

\newcommand{\bb}[1]{{\mathbb{#1}}}
\newcommand{\ca}[1]{{\mathcal{#1}}}
\newcommand{\R}{\mathbb{R}}
\newcommand{\C}{\mathbb{C}}
\newcommand{\bal}{\boldsymbol{\alpha}}
\newcommand{\dt}{\frac{d}{dt}}
\newcommand{\vA}{\varphi}
\newcommand{\de}{\delta}
\newcommand{\K}{\mathcal{J}}
\newcommand{\wK}{\widetilde{\K}}
\newcommand{\subscript}[2]{$#1 _ #2$}

\newcommand{\be}{\begin{equation}}
\newcommand{\ee}{\end{equation}}
\newcommand{\ba}{\begin{aligned}}
\newcommand{\ea}{\end{aligned}}
\newcommand{\bpm}{\begin{pmatrix}}
\newcommand{\epm}{\end{pmatrix}}

\newcommand{\new}[1]{\textcolor{black}{#1}}

\begin{document}

	\title[Decay of solutions of nonlinear Dirac equations: the 2D case]{Decay of solutions of nonlinear Dirac equations: the 2D case}
	\author[S. Herr]{Sebastian Herr} 
	\address{Fakultat f\"ur Mathematik, Universit\"at Bielefeld,  Postfach 10 01 31, 33501 Bielefeld, Germany.}
	\email{herr@math.uni-bielefeld.de}
	\thanks{S.H.: Funded by the Deutsche Forschungsgemeinschaft (DFG, German Research Foundation) -- Project-ID 317210226 -- SFB 1283}
	\author[C. Maul\'en]{Christopher Maul\'en} 
	\address{Fakultat f\"ur Mathematik, Universit\"at Bielefeld,  Postfach 10 01 31, 33501 Bielefeld, Germany. Current address: Departamento de Ingenier\'ia Matem\'atica, Facultad de Ciencias F\'isicas y Matem\'aticas, Universidad de Concepci\'on, Avda. Esteban Iturra sn., Barrio Universitario interior, 4030000 Concepci\'on, Chile.}
	\email{
    chrismaulen@udec.cl}
	\thanks{Ch.Ma.: Funded by the Deutsche Forschungsgemeinschaft (DFG, German Research Foundation) -- Project-ID 317210226 -- SFB 1283, and ANID Fondecyt 1231250}
	\author[C. Mu\~noz]{Claudio Mu\~noz}  
	\address{Departamento de Ingenier\'{\i}a Matem\'atica and Centro
de Modelamiento Matem\'atico (UMI 2807 CNRS), Universidad de Chile, Casilla
170 Correo 3, Santiago, Chile.}
	\email{cmunoz@dim.uchile.cl}
	\thanks{Cl.Mu.: Partially funded by Chilean research grants FONDECYT 1231250 and Basal CMM FB210005.}


	\begin{abstract}
	We study the long-time behavior of small solutions for a broad class of 2D Dirac-type equations with suitable nonlinearities. 
    First, we prove that for nonlinearities with power $p\geq 5$ (massless case) and $p\geq7$ (massive case), any small globally bounded radial solution with vorticity $S\ne -1,0$ decays to zero locally in $L^2_{loc}$, as time tends to infinity. For solutions uniformly bounded in time in a weighted $H^1$ space, this decay result extends to lower powers $p\geq 3$ (massless) and $p\geq5$ (massive).
    Our main results apply to several physical models of current interest, such as the 2D Dirac equation with a honeycomb potential described by Fefferman and Weinstein. Finally, we rule out the existence of small, localized structures such as standing breathers or solitary waves in the 2D regimes considered. 	To prove these results, we introduce new virial identities with a particular algebra that are applied directly to the Dirac model, and without resorting to the nonlinear Klein-Gordon equation.
	\end{abstract}
	\maketitle

	\section{Introduction}
The (linear) Dirac equation was originally formulated as a relativistic counterpart of the Schr\"odinger equation \cite{Dirac}, and it plays a fundamental model within relativistic quantum mechanics. The free 
Dirac equation can be written as 
\begin{equation}\label{linear}
\begin{aligned}
    -i\gamma^{\mu}\partial_\mu \psi + m\psi = 0. 
\end{aligned}
\end{equation}
Here, $\psi:\R\times\R^{n}\to \C^{N}$ denotes a spinor-valued wave function, $m\in \R$ is the mass parameter, and the spinor dimension $N=2^{\lfloor(n+1)/2\rfloor}$. The summation convention 
$\gamma^{\mu}\partial_{\mu}=\sum_{\mu=0}^{n} \gamma^\mu \partial_{\mu}$ is used, with $\partial_{0}\equiv \partial_t$. The Dirac matrices $\gamma^{\mu}$ satisfy the anti-commutation relation
\[
 \gamma^\mu \gamma^\nu + \gamma^\nu \gamma^\mu = 2\,\eta^{\mu \nu},
\]
with the Minkowski metric $\eta=\mbox{diag}(1,-1,\dots,-1)$. 
To capture the dynamics of Dirac fermions, including self-interaction effects \cite{Dirac,Thaller}, various \new{nonlinear versions of this equation} have been proposed and rigorously studied, such as the Soler and Thirring models \cite{soler, T58}, which enjoy invariance under Lorentz boosts.

Defining $\ca{D}_m = -i\gamma^{\mu}\partial_\mu + m$ as the Dirac operator, it is referred to as \emph{massive} when $m\neq 0$ and \emph{massless} when $m=0$. Setting $\beta := \gamma^0$, one observes that $\gamma^0 \ca{D}_m = -i\partial_t + \ca{H}_m$. The associated Hamiltonian is 
$\ca{H}_m = -i\beta\gamma^{j}\partial_j + m\beta$ and we set $\mathcal{H} = \mathcal{H}_0$. With $\bal = (\alpha^1,\dots,\alpha^n)$ and $\alpha^j = \beta\,\gamma^j$ it takes the form
\begin{equation*}
\mathcal{H} = - i\,\bal \cdot \nabla = -i\,\alpha^j \partial_j,
\end{equation*}
further details can be found in \cite{Thaller,OY04}. \new{In dimension $n=2$ we fix the explicit representation, in terms of the Pauli matrices,
\be\label{paulis}
\alpha^1=\sigma_1=\begin{pmatrix}0&1\\ 1&0\end{pmatrix},\qquad
\alpha^2=\sigma_2=\begin{pmatrix}0&-i\\ i&0\end{pmatrix},\qquad
\beta=\begin{pmatrix}-1&0\\ 0&1\end{pmatrix},
\ee
which reproduces system \eqref{eq:2D}.} 

A question of particular interest, closely connected to the long time behavior of solitary waves, is the local energy decay of small solutions in nonlinear Dirac models. This question has been extensively studied in previous works \cite{CF_,Cacciafesta,CS_2016}, specially in the linear case. In \cite{HMM}, the authors studied the long-time behavior of solutions to nonlinear Dirac equations in various regimes. In one dimension massless case, all global solutions decay to zero within an expanding spatial region, thus excluding breathers or soliton-type structures. In the massive one-dimensional case, small symmetric solutions also decay, despite the possible existence of solitary waves. In higher dimensions ($d\geq 3$), decay is obtained outside the light cone under mild assumptions, and local $L^2$ decay is obtained on compact sets for certain nonlinearities including Soler type ones. We recall that the Soler type nonlinearity, appearing in the right-hand side of \eqref{linear},  has the form 
\begin{equation}\label{after}
 f(\psi^{\dagger}\beta \psi) \psi,
\end{equation}
where $f$ is a polynomial real value function such that $f(0)=0$. \new{Equivalently, written as a Hamiltonian evolution equation, the Soler model reads $i\partial_t\Psi=\mathcal{H}_m\Psi+f(\Psi^{\dagger}\beta\Psi)\,\beta\Psi$, which is the usual form in which this nonlinearity appears.} The approach relies on weighted virial identities adapted to the Dirac operator, providing robust dispersive information across different nonlinear settings. Compared with the case where linear operators are treated, nonlinear ones require modifications and careful treatment of nonlinearities since counterexamples to local decay exist in several situations. See Subsection \ref{comparison} for more details.

However, the physically important 2D case was left open. Several new ingredients appear in this case: first of all, the low dimension and the slow decay join the lack of $L^\infty$ control on the solution, making the situation as hard as in 1D, but with additional difficulties because of the lack of good enough embeddings. An example of this fact is that $L^\infty$ control on the solution requires more regularity that just $H^1$, even in the radial case. This is not the situation in 3D, where better estimates are present in the radial case. Additionally, the linear decay $O(1/t)$ might lead to long-range effects of cubic nonlinearities without null-structure.

In 2D, for the spinor $\psi=(\psi_1,\psi_2)$ we consider standard nonlinear Dirac models of the form
\begin{equation}\label{eq:2D}
\begin{aligned}
i(\partial_t\psi_1+\partial_x\psi_2)+\partial_y \psi_2+m\psi_1=&~{} W_1\\
i(\partial_t\psi_2 + \partial_x\psi_1)-\partial_y \psi_1-m\psi_2=&~{} W_2,
\end{aligned}
\qquad \psi_1,\psi_2 \in \C,
\end{equation}
where 
\be\label{non_form1}
(W_1,W_2)=(W_1,W_2)(\psi_1,\overline{\psi_1},\psi_2,\overline{\psi}_2)
\ee
is a polynomial-type nonlinearity depending, in principle, on the four independent variables of the spinor field. We assume that $W$ has the structural gauge property\be\label{gauge}
\ba
& W(e^{iS\theta}\psi_1,e^{-iS\theta}\overline{\psi_1}, i e^{i(S+1)\theta}\psi_2, (-i)e^{-i(S+1)\theta}\overline{\psi_2})\\
&\quad =(e^{iS\theta}W_1, i e^{i(S+1)\theta}W_2) (\psi_1,\overline{\psi_1},\psi_2,\overline{\psi_2}).
\ea
\ee
From the very beginning we see that the form in \eqref{eq:2D} differs from 1D and 3D cases in the sense that in the 1D case the spatial variable, or in the 3D case the radial variable had no spatial counterpart or competitor. In 2D, this is not the case: we have a competition between $x$ and $y$ derivatives of the spinor field $\psi =(\psi_1,\psi_2)$. In what follows, we shall see the physical motivation leading to this model, the recent literature on the subject, 
\new{and, from a physical point of view, its connection with graphene.} One of the main motivations to the present work is to study the nonlinear Dirac equation with a honeycomb \new{nonlinearity}, introduced by Fefferman and Weinstein (see Subsection \ref{sec:2D} for more details and references), given by 
\begin{equation}\label{eq:FW_Dirac}
\begin{aligned}
\partial_t \psi_1 =& -\overline{\lambda}(\partial_x+i\partial_y)\psi_2-ig (\beta_1|\psi_1|^2+2\beta_2 |\psi_2|^2)\psi_1
\\
\partial_t \psi_2 =&-\lambda (\partial_x-i\partial_y)\psi_1 -i g(2\beta_2|\psi_1|^2+\beta_1 |\psi_2|^2)\psi_2,
\end{aligned}
\end{equation}
where $\lambda\in \C\setminus\{0\}$ and $g$ is a real number. 
The above system has the following conserved quantities
\[
\begin{aligned}
M[\psi_1,\psi_2]=&~{}\int_{\R^2} (|\psi_1|^2+|\psi_2|^2),\\
E[\psi_1,\psi_2] =&~{} \Im \int_{\R^2} \psi_2 (\partial_x+i\partial_y)\overline{\psi_1}  
\\&-\frac{g}{4}\int_{\R^2}(\beta_1|\psi_1|^4+4\beta_2|\psi_1|^2|\psi_2|^2+\beta_1|\psi_2|^4),
\end{aligned}
\]
which are the mass and the energy, respectively. It is interesting to notice that the energy is not coercive, a characteristic of Dirac models.

\subsection{2D Dirac models}\label{sec:2D}

In this paper, we shall study \eqref{eq:2D} under a general partial wave decomposition characterized by a parameter $S$, usually denoted as ``vorticity''. In particular, we shall focus our work on a region of parameters where the vorticity is not $0$ or $-1$. In this case, one component of the spinor is purely radially symmetric, making the analysis different from the case $S\neq 0,-1$.

Let us consider polar coordinates $(r,\theta)\in (0,\infty)\times (0,2\pi)$; and a solution of the Dirac equation of the form (see \cite{BF_decay,CKSCL,CBCKS_2018,CS_2016})
\begin{equation}\label{eq:psi1}
\psi(t,x)=\begin{pmatrix} \psi_1 \\ \psi_2 \end{pmatrix}(t,r,\theta)
=\begin{pmatrix}
\phi_1(t,r) e^{iS\theta} \\
i \phi_2(t,r) e^{i(S+1)\theta}
\end{pmatrix}, \quad \mbox{ with } \phi_1,\phi_2 \in \C.
\end{equation}
(For more details see \cite[Proposition 2.1]{CS_2016}.) Notice that the previous decomposition allows us to work within a subspace of the partial wave space, similar to the one presented by Wakano \cite{wakano} in 3D Dirac. Subsequently, after some standard computations, one observes that $\partial_x-i\partial_y=e^{-i\theta}(\partial_r -\frac{i}{r}\partial_{\theta})$ and   $\partial_x+i\partial_y=e^{i\theta}(\partial_r +\frac{i}{r}\partial_{\theta})$. 
Therefore, the 2D nonlinear Dirac equation \eqref{eq:2D}, when restricted to solutions of the form \eqref{eq:psi1}, becomes equivalent to the model
\begin{equation}\label{eq:2Dr}
\begin{aligned}
i\partial_t \phi_1=&~{} \left(  \partial_r   +\frac{(S+1)}{r} \right)\phi_2-m\phi_1+W_1,
\\
i\partial_t \phi_2=&~{} -\left(\partial_r -\frac{S}{r}\right)\phi_1+m\phi_2 +W_2,
\end{aligned}
\end{equation}
which is a reduced $2\times 2$ radial complex system, where, from \eqref{non_form1},
\be\label{deco_complex}
\ba
(\phi_1,\phi_2)= &~{} (\phi_{11}+i\phi_{12},\phi_{21}+i\phi_{22}) \in \C^2, \\
W=(W_1,W_2)=&~{} (W_{11}+iW_{12},W_{21}+iW_{22}) \in \C^2.
\ea
\ee
In order to get \eqref{eq:2Dr}, we have assumed that the nonlinearity has the structural gauge property \eqref{gauge}.
Notice that the structural gauge property introduced  is present, for instance, in the case of odd power nonlinearities. The property \eqref{gauge} makes \eqref{eq:2Dr} independent of the angular variable, and it is indeed present in practice, see e.g. the Soler's case\footnote{In this case, the nonlinearity is given by $W_1 = -f(\psi^{\dagger}\beta \psi) \psi_1\quad W_2=f(\psi^{\dagger}\beta \psi) \psi_2$} \cite{soler} or \eqref{HC} below in the 2D setting. Now, to simplify some of the forthcoming computations, let us define
\be\label{mod1}
|\phi|^2:= \phi_{11}^2 +\phi_{12}^2 +\phi_{21}^2 +\phi_{22}^2,
\ee
and
\be\label{mod2}
 |\nabla \phi|^2:= (\partial_r \phi_{11})^2 +(\partial_r \phi_{12})^2 +(\partial_r \phi_{21})^2 +(\partial_r \phi_{22})^2.
\ee
Assume additionally that the non-linearity is of polynomial type at small scales: there exists $C>0$ and $p_{*}>1$ such that
\begin{equation}\label{eq:W_2D}
|W_1|+|W_2|\leq C|\phi|^{p}, \mbox{ for } p\geq p_{*}, \quad |\phi|<1.
\end{equation}
\new{We emphasize that, beyond \eqref{eq:W_2D}, the proofs only use the companion bound on the derivative of the nonlinearity,
\begin{equation}\label{eq:dW_2D}
|\partial_r W_1|+|\partial_r W_2|\leq C|\phi|^{p-1}|\nabla\phi|, \qquad |\phi|<1,
\end{equation}
which holds for instance whenever $W$ is a polynomial in $(\phi_1,\overline{\phi_1},\phi_2,\overline{\phi_2})$ vanishing to order $p$ at the origin.} 

We remark that condition \eqref{eq:W_2D} covers classical models known to possess solitary waves, such as the Dirac equation with a honeycomb \new{nonlinearity} \cite{Honey}, namely the nonlinear model introduced by Fefferman and Weinstein, see \eqref{eq:FW_Dirac}, when dealing with waves in 2D honeycomb structures (the linear case was considered in \cite{FW1,FW2}). \new{Here we use the term \emph{honeycomb nonlinearity} rather than \emph{honeycomb potential}: Fefferman and Weinstein obtain \eqref{eq:FW_Dirac} as an effective model from the cubic nonlinear Schr\"odinger equation with a potential having the symmetries of the honeycomb lattice, for excitations near a Dirac point.} 

Indeed, for this case, the nonlinearity $W=(W_1 ,W_2)$ is given by:
\begin{equation}\label{HC}
W_1 = (\beta_1|\phi_1|^2+\new{2}\beta_2|\phi_2|^2)\phi_1, \quad W_2 =(\new{2}\beta_2|\phi_1|^2+\beta_1|\phi_2|^2)\phi_2,
\end{equation}
where $(\beta_1,\beta_2)$ are positive real-valued parameters associated to the Floquet-Bloch model. The motivation behind this model is as follows. In \cite{Honey}, the authors aimed to study the propagation of waves on honeycomb structures, such as those present in 
\new{graphene}, a 2D structure where the carbon atoms are arranged in a
honeycomb lattice. With this in mind, the authors started with the non-relativistic Schr\"odinger equation with a honeycomb lattice potential and formally derived a massless, nonlinear 2D Dirac system that governs the effective dynamics (see (4.4)-(4.5) in \cite{Honey}). This model has been widely studied in a series of works \cite{Honey,Honey2,FW1,FW2,Transverse_Pelinovsky} and extended to deformed honeycomb lattice structures (see (5.2a)-(5.2b) in \cite{AY}).

The second motivation in our setting is given by recent works dealing with critical Dirac elliptic models. In a series of papers, W. Borrelli and R. Frank have studied the existence of ground state bubbles and killing spinors for the critical Dirac equation on dimensions greater than or equal to two with several type potentials (see \cite{bubbles,B_loc} and the references therein). In particular, in \cite{B_loc}  Borrelli showed the existence of infinitely many non square-integrable stationary solutions for a family of massless Dirac equations in 2D, using a particular radial ansatz of the form \eqref{eq:psi1} with $S=0$. Later, Borrelli and Frank in \cite{BF_decay} proved sharp pointwise decay in $\R^n$ with $n\geq 2$ of the solutions to the critical Dirac equation on compact spin manifolds. In the 2D case, for a cubic nonlinearity of the type $|\psi|^2\psi$, the authors found that the obtained decay estimates are sharp for "ground state solutions". In contrast, the "excited state solutions" in general exhibit a faster decay. In particular, Theorem 1.4 in \cite{BF_decay} describes the asymptotic behavior of static solutions, both as $r\to \infty$ and near $r=0$, to \eqref{eq:2Dr} with a Honeycomb potential. The authors noted that the behavior changes depending on the sign of the term $S+1/2$, where $S$ is the vorticity parameter. Moreover, the authors constructed an explicit solution in the form of a rational polynomial when the parameters are $\beta_1=1$ and $\beta_2=1/2$ (see \cite[Theorem 7.1]{BF_decay}). \new{With this choice, and recalling the factor $2$ in \eqref{eq:FW_Dirac}, the honeycomb nonlinearity reduces to the Kerr nonlinearity $|\psi|^2\psi$.}

The global well-posedness problem of the nonlinear Dirac equation is challenging due to the lack of a definite sign in the Hamiltonian. Most works address the massive and the massless cases separately. The literature is extensive across both 1D and higher dimensions; here, we highlight the most relevant results for the 2D case. Firstly, we shall recall that the cubic massless Dirac equation is invariant under the scaling $\psi(t,x)\to \lambda^{1/2}\psi(\lambda t, \lambda x)$, thus the scale invariant regularity is $s_c=(n-1)/2$, where $n$ is the dimension of the problem.  Pecher \cite{pecher1,pecher2} proved that the cubic 2D Dirac equation is locally well-posed in $H^s$ for the almost critical exponent $s > 1/2$.
Later,  Bournaveas and Candy \cite{BC_2d},  using the Fierz identity and the null structure obtained,  proved that massless cubic Dirac equation is globally well-posed  for small data in the \emph{scale invariant space} $\dot{H}^{1/2}(\R^2)$.  Bejenaru and the first author \cite{BH}, established the global well-posedness and scattering for the cubic Dirac equation with small initial data in the critical space $\dot{H}^{1/2}$.
Later, in \cite{GS_hartree}  local and global well-posedness for the Dirac equation with Hartree type-nonlinearity was proved in $H^s$ for $s\geq0$. Recently, Candy and the first author \cite{CH23}, for the cubic Dirac equation in 2D and 3D, offer a unified treatment of the massive and massless cases, showing their intrinsic connections and obtaining convergence in the massless and the non-relativistic limit. They employed a bilinear Fourier restriction method and atomic function spaces to prove the global well-posedness of the Cauchy problem for small initial data. 

The nonlinear Dirac equation is well-known to admit solitary wave solutions for a wide range of nonlinearities,  see \cite{CBCKS_2018,CKSCL,bubbles} for existence and properties in the 2D case. \new{For the existence of solitary waves in the massive case, see also the variational constructions of Cazenave--V\'azquez \cite{CazenaveVazquez} and Esteban--S\'er\'e \cite{EstebanSere} carried out for $d=3$, and it is possible that some of their arguments carry over to $d=2$.}
 
In \cite{CBCKS_2018,CKSCL}, the authors studied the 2D massive nonlinear Dirac-type equation and \new{numerically} analyzed the spectral stability of its solitary waves and vortex solutions.
 For the cubic nonlinearity, the states of higher vorticity are generically unstable and split into lower charge vortices in a way that preserves the total vorticity. 
Conversely, for the quintic nonlinear 2D Dirac equation, solutions with vorticity $S=0$ are potentially stable. Moreover, they observed that with a quintic nonlinearity, the instabilities caused by radially symmetric perturbations cause the density width (and amplitude) to oscillate, leading to a "breathing" structure, but there is no collapse. In the case of the Kerr type nonlinearity, also known as mean field interaction  (i.e. $|\psi|^2\psi$), which usually appears in the description of Bose-Einstein condensates, Borrelli in \cite{B_SS_Kerr} proved the existence of smooth, exponentially decaying static solutions. Moreover, the same author in \cite{B_Sym_2d} obtained the existence of infinitely many symmetric smooth, exponentially decaying solutions for the 2D Dirac equation with Honeycomb potential, for $S\neq 0$ and $\psi(0)=0$.

\subsection{Main results}\label{sec:main}

Now we shall assume $S\neq 0,-1.$ Recall that in the case of zero vorticity and non zero vorticity ($S=0$ and $S\neq 0$, respectively) the existence of static solutions has been proved in \cite{B_SS_Kerr,B_Sym_2d}, and their classification and description were given  in \cite{BF_decay}. We will return to this point for a detailed discussion after establishing the main results of the present work.

Let $\delta>0$ be a fixed parameter. Recall the notation for $|\phi|^2$ and $|\nabla \phi|^2$ introduced in  \eqref{mod1}-\eqref{mod2}. Define the following weighted space of radial functions
\begin{equation}\label{eq:weight}
\begin{aligned}
E(\de) =\left\{ \Psi\in H^1_{\mathrm{rad}} ~{} \big| ~{}  \|\Psi \|_{E(\de)} := \| \langle r\rangle ^{\delta/2}\nabla \Psi\|_{L^2}+\| \langle r\rangle^{\delta/2}\Psi \|_{L^2}<\infty \right\}.
\end{aligned}
\end{equation}
\new{where, and throughout this work, for radial functions the $L^2$-norm is defined by
\be\label{defn_L2}
\|f\|_{L^2}^2=\int_{0}^{\infty} f^2(r)\,r\,dr.
\ee}
Notice that \eqref{defn_L2} takes into account the fact that we are working in two dimensions. Having this fact in mind, we can now establish our first result for the 2D Nonlinear Dirac equation.
\begin{thm}\label{thm:2D_m-}
Assume that $W=(W_1,W_2)$ obeys \eqref{eq:W_2D} with thresholds
\begin{enumerate}[label=\emph{(\subscript{L}{\arabic*})}]
\item\label{thm:-_m-_w}   $p_{*}=3$ in the massless case;
\item\label{thm:-_m+_w}  $p_{*}= 5$ in the massive case.
\end{enumerate}
Let $\delta>0$ be fixed. There exists $\varepsilon>0$ such that the following is satisfied: Let
\[
{\bf \phi}=(\phi_1,\phi_2)\in C^1_{loc} \left( \R: L^2(\R^2;\C)^2 \right)\cap C_{loc} \left( \R: E(\delta)\right),
\] 
be any radial global solution to the  2D Dirac equation \eqref{eq:2D}-\eqref{eq:psi1}, with vorticity $S\in \bb{Z}\setminus\{-1,0\}$, and small, in the sense that $\sup_{t\geq 0} \|{\bf \phi}(t)\|_{E(\delta) \cap L^\infty} <\varepsilon$. 
Then, for  any $R>0$, it holds
\begin{equation}\label{limit2D_m-}
\lim_{t\to \infty} \| {\bf \phi} (t)\|_{L^2(B(0,R))}=0.
\end{equation}
Therefore, \new{there exist} neither small standing solitons nor small breather solutions for the Dirac equation \eqref{eq:2D} within any compact set of $\mathbb R^2$. \end{thm}

Notice that the condition $\sup_{t\geq 0} \|{\bf \phi}(t)\|_{E(\delta)} \ll 1$ for any $\delta>0$ does not ensure $L^\infty$ control on the solution, therefore, some $L^\infty$ control is also necessary. 

The hypothesis $\sup_{t\geq 0} \|{\bf \phi}(t)\|_{L^\infty} \ll 1$ is ensured if, for instance, one has $\sup_{t\geq 0} \|{\bf \phi}(t)\|_{H^{1+}} \ll 1$.

It is interesting to compare \eqref{limit2D_m-} with the results proved in \cite{BF_decay}. In this paper, Borrelli and Frank established a classification for the ground state solutions as well as the excited states for the massless 2D Dirac equation with honeycomb potential (characterized by cubic nonlinearities), using the ansatz \eqref{eq:psi1}. In particular, they provided explicit asymptotic behavior for such solutions in terms of the vorticity $S\in \bb{Z}$, and the quantity $S+1/2$. For example, when $S+1/2>0$, $(\phi_1,\phi_2)$ behaves like $(r^{S},r^{3S+1})$ near zero, and like $(r^{-(3S+2)},r^{-(S+1)})$ as$r\to \infty$. On the other hand, when $S+1/2<0$, $(\phi_1,\phi_2)$ behaves like $(r^{-(3S+2)},r^{-(S+1)})$ near zero, and like $( r^{S},r^{3S+1})$ as $r\to \infty$ (see \cite[Theorem 1.4]{BF_decay} for more details). Recalling \eqref{eq:FW_Dirac} and choosing $\beta_1=2~{},\beta_2=1$, the authors obtained an explicit form for the excited states, under the condition \new{that} $r^{1/2}\phi_1(r)$ and $r^{1/2}\phi_2(r)$ vanish as  $r\to 0$. These real solutions are given by
\[
\phi_1(r)=\tau \lambda^{-1/2}V \left( \frac{r}{\lambda} \right) \quad\mbox{ and }\quad \phi_2(r)=\sigma \lambda^{-1/2}U \left( \frac{r}{\lambda} \right),
\]
where
\[
V(r)=\frac{ (2|2S+1|)^{\frac12} r^{-(S+1)}}{r^{2S+1}+r^{-(2S+1)}}\quad \mbox{and} \quad U(r)=r V(r),
\]
for \new{any} $\lambda>0$ (i.e. $\lambda$ is arbitrary), with $\sigma=\tau=1$ if $2S+1>0$, and $\sigma=-\tau=1$ if $2S+1<0$ (see \cite[Theorem 7.1]{BF_decay}).  We observe that the above solutions are not sufficiently small in the $L^\infty\cap L^2$ norm, independent of the value of $\lambda$. Working in a ``weak'' weighted Sobolev space (see \eqref{eq:weight} and \eqref{eq:varphi}), we are able to treat the critical 2D Dirac equation, i.e., the cubic massless, and the quintic pure power nonlinearity in the massive case. The method, as we shall explain in detail below, does not use the Klein-Gordon trick, but instead attacks directly Dirac equations. The additional decay assumption ensures that the virial functionals are well-defined (see \eqref{eq:varphi}).  While this condition slightly deviates from the classical Sobolev framework, it is primarily believed that it is a technical condition. In any case, Theorem \ref{thm:2D_m-} considers solutions outside the Borrelli-Frank framework. 

Let us remark that the decay properties of the 2D \emph{linear} Dirac equation with a potential have also been extensively studied. Fundamental decay results via virial identities were proved in \cite{BDAF}, in the case of linear Dirac models with magnetic potentials. Cacciafesta and S\'er\'e \cite{CS_2016}, using a spherical harmonics decomposition in the 2D and 3D cases, 
 proved local smoothing estimates for the massless Dirac equation with a Coulomb potential. Later,  Cacciafesta and Fanelli \cite{CF_} extended this work, proving local smoothing and weighted Strichartz estimates for the Dirac equation with an Aharonov-Bohm potential. Furthermore, Erdo\u{g}an and Green \cite{EG}, studied dispersive estimates for the 2D Dirac equation with a potential, obtaining that the Dirac evolution satisfies a $t^{-1}$ decay rate as an operator from the Hardy space $H^1$ to $BMO$, the space of functions of bounded mean oscillation. In \cite{Cacciafesta}, the author obtained a virial identity for an $n$-dimensional linear Dirac equation perturbed with a magnetic potential, and then  used it to deduce smoothing and Strichartz estimates for the case where $n\geq 3$. The work \cite{DAFS} studied the spectral properties  of the $n$-dimensional massive Dirac operator, for  $n\geq 2$, perturbed by a potential $V$. 
They proved that the eigenvalues are contained in the union of two disjoint disks in the complex plane, provided $V$ is sufficiently small in certain mixed norms. Furthermore, they showed that under the same smallness condition, the discrete spectrum is empty in the massless case.

\medskip

Now we establish a second, less demanding result for the case where the weighted condition \eqref{eq:weight} is no longer required, and we only work in the space $H^1\cap L^\infty.$ The lack of $L^\infty$ control for data only in $H^1$ (even in the radial case) motivates the use of the mixed $H^1-L^\infty$ control. 

\begin{thm}\label{thm:2D_m+}
 Assume that $W=(W_1,W_2)$ obeys \eqref{eq:W_2D} with
\begin{enumerate}[label=$($\subscript{H}{\arabic*}$)$]
\item \label{thm:+_massless_Strong} $p_{*}=5$ in the massless case:
\item \label{thm:+_massive}  $p_{*}= 7$ in the massive case.
\end{enumerate}
There exists $\varepsilon>0$ such that the following is satisfied: Let 
\[
{\bf \phi}=(\phi_1,\phi_2)\in C^1_{loc} \left( \R: L^2(\R^2;\C)^2 \right)\cap C_{loc} \left( \R: H^1(\R^2;\C)^2\right),
\] 
be any radial global solution to the 2D Dirac equation \eqref{eq:2D}-\eqref{eq:psi1}, with vorticity $S\in \bb{Z}\setminus\{-1,0\}$, and $\sup_{t\geq 0} \|{\bf \phi}(t)\|_{H^1\cap L^\infty} <\varepsilon$.
Then, for  any $R>0$,
\begin{equation}\label{limit2D}
\lim_{t\to \infty} \| {\bf \phi} (t)\|_{L^2(B(0,R))}=0.
\end{equation}
Therefore, \new{there exist} neither small standing solitons nor small breather solutions for the Dirac equation \eqref{eq:2D} within the ball $|x|<R$.
\end{thm}

Recall that the above results complement those proved in \cite{HMM},  for the exterior region of the multidimensional ``light cone''. The system studied there was
\begin{equation}\label{eq:dirac_1}
\begin{aligned}
i\partial_{t} \psi =&~{}\mathcal{H} \psi+ (m+V(x,\psi)) \beta \psi\\
\psi(0,x)=&~{}\psi_{0}(x),
\end{aligned}
\end{equation}
for the case of general dimension $n\geq 1$.  Fix $j\in \{1,\dots, n\}$. Define the time-depending interval 
\begin{equation}\label{eq:Ijb}
\mathcal{I}_{b}^{j}(t)=\big\{ x\in \R^{n}~{} \vert {}~  |x_j|\geq (1+b)t \big\}, \quad t >2, \quad b>0.
\end{equation}

\begin{thm}[{\cite[Theorem 1.5]{HMM}}]
Let $ m \in \R$ and $\psi\in C(\R; L^2(\R^n;\C^N))$ be a global solution to the Dirac equation \eqref{eq:dirac_1} with a potential $V$ such that
 \begin{equation*}
    V(x,z)\in \R, \mbox{ with } (x,z)\in\R^n\times \C^N,
 \end{equation*}
 and $b$ is an arbitrary positive number.
Then, on the region $\mathcal{I}_{b}^j(t)$ (see \eqref{eq:Ijb}), there is a strong decay to zero of the charge, i.e.  
\begin{equation*}
\lim_{t\to\infty}  \underset{~{}\mathcal{I}^j_{b}(t)}{\int} \psi^{\dagger}(t,x)\psi(t,x)dx
=0 \quad \mbox{ for all } j\in\{1,\dots,n\}.
\end{equation*}
Moreover, there does not exist any non-decaying in time solution for the Dirac equation concentrated inside the region $\mathcal{I}^j_b(t)$.
\end{thm}

In other words,  for $V:\R^n\times\C^{N} \to \R$,  every solution $\psi$ belonging to $L^2$ converges to zero in the $L^2$ norm in the exterior region of the multidimensional ``light cone''.

\subsection{Comparison with the 3D case}\label{comparison} 
Notice that the system \eqref{eq:2Dr} looks similar to the one studied in \cite{HMM}. By replacing the factors $S+1$ and $S$ with 2 and 0, respectively, one recovers the 3D radial case of the nonlinear Dirac equation in \cite{HMM}. However, there is no direct translation of \eqref{eq:2Dr} into a corresponding 3D radial version of the nonlinear Dirac equation. On the other hand, in the model studied in the present work, intrinsic difficulties arise from the weak spatial decay present in two dimensions, which makes the 3D case comparatively easier.

In comparison to the 3D case, the weaker decay in 2D forces us to impose stronger hypotheses in order to treat the cubic and quintic nonlinearities in the massless and massive cases, respectively (see \eqref{eq:weight} in Theorem \ref{thm:2D_m-}). The hypotheses mentioned above can be interpreted as technical requirements arising only in the low-order nonlinear cases; specifically, 
\new{an integrability condition of order $r^{\delta}$, for some $\delta>0$, is required.} We believe that this condition could be improved to order $\log(r+1)$. However, we must remark that for higher-order nonlinearities, namely those greater than or equal to quintic and septic in the massless and massive cases, respectively, the additional hypothesis is no longer required. Under these conditions, we recover the analogous version of the established theorem in the 3D case for the cubic nonlinearity in \cite{HMM}.

\subsection{Method of proof}\label{MoP}

Mathematically speaking, \new{virial identities motivated by decay} in the radial 2D case, as well as in the 3D case, present better spectral properties than in the 1D case \cite{MoMu,HMM}, making more quantitative the classical argument ''the bigger the dimension, the better the linear decay''. This is reflected in the fact that one can construct suitable virial weights under which local decay is achieved, and this procedure fails in 1D by classical counterexamples (breathers). In this paper, after a suitable decomposition of $\phi$ in components $(\phi_{11},\phi_{12},\phi_{21},\phi_{22})$, we will consider 2D radial virial functionals of the form 
\[
\begin{aligned}
\K_{1}=&~{}\int \left[ \vA  \partial_r \phi_{11} +\frac12 \vA ' \phi_{11} \right]\left(  \left( \partial_r +\frac{S+1}{r}\right)  \phi_{22}-m\phi_{12}  \right),
\\
\wK_{1}=&~{}\int \left[ \vA  \partial_r \phi_{22} +\frac12 \vA ' \phi_{22} \right]\left( \left( \partial_r -\frac{S}{r}\right) \phi_{11}-m\phi_{21}
 \right),
\end{aligned}
\]
with vorticity $S\neq -1,0$ and  weight $\vA  =\frac{r^{3}}{(1+r)}$ (in the case of strong nonlinearity). This weight is motivated by several key interactions: at short distances 
\new{it is stronger than the natural energy weight (the volume element $r\,dr$) by two powers of $r$, that is, it behaves like $r^{3}$ instead of $r$ near the origin}, and at large distances it joins the natural decay in the energy space. See \cite{AC_NLS_NonExistence,ACKM} for other previous works using virial identities of this type. 
\new{They are complemented} by two additional virials $\K_{2}$ and $\wK_{2}$.  The combination of these four functionals, together with a carefully chosen weight $\vA$ which exhibits a strong decay at $r=0$ (order $r^3$) but an additional behaves at infinity like $r$ (the natural volume element in 2D),  will help us to describe the asymptotic behavior locally in space dynamics of small globally bounded 2D Dirac massless and massive waves. In this framework, it is proved \eqref{limit2D}. A remarkable aspect of Theorem \ref{thm:2D_m+} is that, although the massive case is generally expected to have better decay properties, our analysis shows that it actually requires a higher-order nonlinearity to conclude. This may be related to the possible existence of internal modes introduced by the weights used in virial identities (which are large perturbations in several cases) and that may disturb the natural decay.
\new{
The strategy can be summarized as follows. We combine the four virial functionals into $\K$ and compute its time derivative, for the high-order nonlinearities treated in Section~\ref{sec:thm+}, we establish the uniform-in-time pointwise bound (see \eqref{final_K+})
\[
-\frac{d}{dt}\K(t)\gtrsim \int_0^\infty\frac{r^2}{(1+r)^3}|\nabla\phi(t,r)|^2\,dr
+\int_0^\infty\frac{1}{(1+r)^2}|\phi(t,r)|^2\,dr,
\]
where the implicit constants are independent of $t$. Since $\K(t)$ remains uniformly bounded for the solutions we consider (see the bounds in Subsection~\ref{sec:thm+}), integrating the above inequality over $[0,\infty)$ shows that the right-hand side belongs to $L_{t}^2(I;L^2_r(0,\infty))$. Consequently, the quadratic terms must vanish along a sequence $t_n\to\infty$. In Section~\ref{Section:2D} we relate this vanishing to the weighted mass functional $\mathcal H(t)$ defined in \eqref{mathH}, which allows us to conclude that $\|\phi(t_n)\|_{L^2_{\rm loc}}\to 0$. A final monotonicity argument then yields the full local $L^2$ decay to zero as $t\to\infty$, proving \eqref{limit2D_m-} and \eqref{limit2D}. The low-order case in Section~\ref{sec:thm-} follows the same pattern, using the analogous weighted bound \eqref{final_K}.
}
\begin{rem}\label{comS}
\new{In the proof, the assumption $S\neq0,-1$ plays a key role in the virial identities. The quadratic part controlling each pair of components carries a coefficient proportional to $4K^2-1$, with $K=S$ for $(\phi_{11},\phi_{12})$ and $K=S+1$ for $(\phi_{21},\phi_{22})$ (see Section~\ref{sec:thm+} and \eqref{quadratic_part}). This coefficient is strictly positive precisely when $K\neq0$, i.e.\ $S\neq0$ and $S\neq-1$; for $S=0$ or $S=-1$ one of these coefficients becomes $-1$, and the coercive sign that drives the decay argument is lost. This is the analytic reason behind the restriction $S\neq0,-1$.}
\end{rem}

\subsection*{Organization of this paper} 
This paper is organized as follows: Section \ref{sec:virial} presents some preliminary identities and estimates to deal with new virial identities introduced in this paper. Section \ref{sec:thm-} is devoted to giving the first steps on the proof of Theorem \ref{thm:2D_m-}. Section \ref{sec:thm+} deals with the first computations on the proof of Theorem \ref{thm:2D_m+} in the massless case. Finally, in Section \ref{Section:2D} we prove Theorem \ref{thm:2D_m-} and Theorem \ref{thm:2D_m+}.

\subsection*{Acknowledgments} 
Cl. Mu. thanks to CRC 1283 at Bielefeld University (Germany) for their support. He is also grateful to IMAG-BIRS and the Departamento de Matem\'aticas Aplicadas in Granada, Spain, for their hospitality, where part of this work was completed.

We thank the anonymous referees for their careful reports and valuable suggestions.

\section{Virial identities}\label{sec:virial}

\subsection{Preliminaries}
First of all, we claim the following useful auxiliary result. See
\cite{HMM} for a similar proof\new{; the proof of Claim~\ref{Cl1} is also included in Appendix~\ref{app:Cl1}}. \new{Here and in the sequel, unless otherwise stated, all integrals are taken over the half-line $(0,\infty)$ in the radial variable $r$ (without the $2$D volume factor $r\,dr$), as further discussed in the Remark on the absence of volume terms below.}
\begin{claim}\label{Cl1}
Let $K\in\mathbb R$, $\vA $ be a smooth bounded function such that $\vA (0)=0$ and let $f\in H^1$. It holds that
\[
\begin{aligned}
 \int \left[ \vA  \partial_r f +\frac12 \vA ' f \right] &   \left(\partial_r+\frac{K+1}{r} \right)\left(\partial_r -\frac{K}{r}\right)f \\
  &\quad =-\frac12 \int \left(2\vA' -\frac{2}{r} \vA \right)(\partial_r f)^2
 \\&\qquad 
 - \frac12 \int \bigg( 
 			2\vA \frac{K^2}{r^3}
 			+ \vA''  \frac{1}{2r}
			 -\vA '  \frac{1}{2r^2}
			- \frac12   \vA'''  
		\bigg) f ^2,
\end{aligned}
\]
and
\[
\begin{aligned}
 \int \left[ \vA  \partial_r f +\frac12 \vA ' f \right] &   \left(\partial_r-\frac{K}{r} \right)\left(\partial_r +\frac{K+1}{r}\right)f \\
=&\quad  -\frac12  \int    \left( 2\vA' -\vA \frac{2}{r} \right)(\partial_r f )^2
\\&\quad 
		-\frac12 \int    \bigg(  2\vA \frac{(K+1)^2}{r^3} 
		-  \frac12  \vA '''  
		+ \vA ''   \frac{1}{2r} 
		-\vA '   \frac{1}{2r^2}  
		 \bigg)  f^2.
\end{aligned}
\]
\end{claim}

The proof is standard, but it is included in Appendix \ref{app:Cl1} for completeness.

\medskip

Now, we recall a well-known result which will help us deal with the nonlinear terms; a version of this Lemma has been used in \cite{MoMu, HMM} in the 3D case.

\begin{lem}[Strauss]\label{lem2d}
Let $u\in H^1(\R^2)$ be a radial function. Then $u(r)\in L^2(0,\infty)$ and 
it holds
\[
\sup_{r\geq 0} |r^{1/2} u(r)|\leq C\|u\|_{H^1(\R^2)}.
\]
Moreover, if $\sup_{t\geq 0} \| u(t)\|_{H^1\cap L^\infty}  \leq \epsilon$, then
\[|u(t,r)| \lesssim \epsilon (1+r)^{-1/2}.\]
\end{lem}

\begin{rem}
In 2D, Strauss' lemma provides a weaker decay estimate compared to the 3D setting.
This lack of strong spatial decay forces us to impose stronger assumptions than in the 3D setting. This dimensional limitation is intrinsic to the structure of radial Sobolev embeddings in low dimensions.
\end{rem}

\subsection{Virial Identities}

Let us consider global radial solutions to \eqref{eq:2Dr}. Decompose $\phi_{j}=\Re \phi_j+i\Im \phi_j=:\phi_{j1}+i\phi_{j2}$. Rewriting the real and imaginary parts of the above system and rearranging equations, we obtain the following complex system
or corresponding  real-valued $4\times 4$ system
\begin{equation}\label{eq:SD2}
\begin{aligned}
\partial_t \phi_{11} = &~{}   \left( \partial_r +\frac{S+1}{r}\right)  \phi_{22}-m\phi_{12}   +W_{12} \\
\partial_t \phi_{22} = &~{}  \left( \partial_r -\frac{S}{r}\right) \phi_{11}-m\phi_{21} -W_{21} \\
\partial_t \phi_{12} = &~{} - \left( \partial_r +\frac{S+1}{r}\right) \phi_{21}+m\phi_{11}-W_{11} \\
\partial_t \phi_{21} = &~{}- \left( \partial_r -\frac{S}{r}\right)  \phi_{12}+m\phi_{22}+W_{22}.
\end{aligned}
\end{equation}

To prove decay as in \eqref{limit2D}, we shall introduce new virial functionals adapted to \eqref{eq:SD2}. These functionals are reminiscent of previous works on NLKG models, Schr\"odinger and wave maps, see e.g. \cite{KMM_Nonexistence_KG,EM20,AleMau,MoMu}, and the recent work on the Dirac equation in 1D and 3D \cite{HMM}.

Let us consider the virial functionals
\begin{equation}\label{eq:K1}
\begin{aligned}
\K_{1}=\int \left[ \vA  \partial_r \phi_{11} +\frac12 \vA ' \phi_{11} \right]\left(  \left( \partial_r +\frac{S+1}{r}\right)  \phi_{22}-m\phi_{12}  \right),
\end{aligned}
\end{equation}
and
\begin{equation}\label{eq:tK1}
\begin{aligned}
\wK_{1}=\int \left[ \vA  \partial_r \phi_{22} +\frac12 \vA ' \phi_{22} \right]\left( \left( \partial_r -\frac{S}{r}\right) \phi_{11}-m\phi_{21}
 \right).
\end{aligned}
\end{equation}
Similar to the 3D case, achieving complete control over the terms necessitates a second set of virial functionals for reconstructing the full local norm of the solution. Let
\begin{equation}\label{eq:def_K2}
\begin{aligned}
\K_{2}=\int \left[ \vA  \partial_r \phi_{12}+\frac12 \vA '  \phi_{12}\right] \left(   \left( \partial_r +\frac{S+1}{r}\right)\phi_{21}-m\phi_{11}\right),
\end{aligned}
\end{equation}
and
\begin{equation}\label{eq:def_tK2}
\begin{aligned}
\widetilde{\K}_{2}=\int \left[ \vA  \partial_r \phi_{21}+\frac12 \vA '  \phi_{21}\right]\left( \left( \partial_r -\frac{S}{r}\right)  \phi_{12}-m\phi_{22}\right).
\end{aligned}
\end{equation}

These are virial functionals adapted to the 2D setting. Notice that integrals are taken in the interval $(0,\infty)$. Therefore, the condition $\vA (0)=0$ is important to avoid undesirable boundary terms. See \cite{MoMu,HMM} for a detailed discussion on this point.

\begin{rem}[On the absence of volume terms] The reader may notice that each virial term in \eqref{eq:K1}-\eqref{eq:def_tK2} does not consider the volume term, which in 2D corresponds to a coefficient $r$ in front of $dr$. Although this will not change the main result, it has an important meaning in the forthcoming computations, in the sense that we let the problem  tell us which is the best weight for getting decay, without being influenced by the geometry of the problem. This choice has good and bad consequences, the bad one is that one later has to recover the energy estimates in terms of the weight $rdr$, however, the good point is to find the correct virial identity, and this goal is more important than recovering energy estimates, which are always performed later than virial estimates.
\end{rem}

We can now prove the first set of virial identities, which help us to control the components $\phi_{11}$ and $\phi_{22}$.
\begin{prop}\label{prop:K1}
Let $\vA $ be a smooth bounded function such that $\vA (0)=0$ \new{and $\vA'(0)=0$}. Under the assumptions on $\vA $ it holds that
\[
\begin{aligned}
\dt \K_1 
=&~{}
\int \left[ \vA   \partial_r W_{12} +\frac12 \vA ' W_{12}  \right]\left(  \left( \partial_r +\frac{S+1}{r}\right)  \phi_{22}-m\phi_{12} \right)
\\&\quad
+\int \left[ \vA  \partial_r \phi_{11} +\frac12 \vA ' \phi_{11} \right]
			\bigg( 
				- \left( \partial_r +\frac{S+1}{r}\right)  \left(W_{21}\right)
					+m W_{11}		
			 \bigg)
\\
&\quad 	
-\frac12 \int \left(2\vA' -\frac{2}{r} \vA \right)(\partial_r \phi_{11}	)^2
 \\&\qquad
 - \frac12 \int \bigg(
 			2\vA \frac{\new{S}^2}{r^3}
 			+ \vA''  \frac{1}{2r}
			 -\vA '  \frac{1}{2r^2}
			- \frac12   \vA'''
		\bigg) \phi_{11}^2,
\end{aligned}
\]
and
\[
\begin{aligned}
\dt \widetilde{ \K}_1  
=&~{}
-\int \left[ \vA  \partial_r \left( W_{21} \right) +\frac12 \vA ' \left(  W_{21} \right)\right]\left( \left( \partial_r -\frac{S}{r}\right) \phi_{11}-m\phi_{21}  \right)
\\& + \int \left[ \vA  \partial_r \phi_{22} +\frac12 \vA ' \phi_{22} \right]  \left( 
				 \left( \partial_r -\frac{S}{r}\right) \left( W_{12} \right)
				-m  W_{22} 
				 \right)
\\& 
 -\frac12  \int    \left( 2\vA' -\vA \frac{2}{r} \right)(\partial_r  \phi_{22}   )^2
\\&
		-\frac12 \int    \bigg(  2\vA \frac{(S+1)^2}{r^3} 
		-  \frac12  \vA '''  
		+ \vA ''   \frac{1}{2r} 
		-\vA '   \frac{1}{2r^2}  
		 \bigg)   \phi_{22}  ^2.
\end{aligned}
\]
Additionally,

\begin{equation*}
\begin{aligned}
\dt \K_{2}
=&- \int \left[ \vA  \partial_r \left( W_{11}\right)
+\frac12 \vA ' W_{11} \right] \left(   \left( \partial_r +\frac{S+1}{r}\right)\phi_{21}-m\phi_{11}\right)
\\&
-\int \left[ \vA  \partial_r \phi_{12}+\frac12 \vA '  \phi_{12}\right]
  \left(   
			\left( \partial_r +\frac{S+1}{r}\right) \left(-W_{22}\right)
			+m W_{12}
			\right)
\\&
+\frac12 \int \left(2\vA' -\frac{2}{r} \vA \right)(\partial_r  \phi_{12})^2
 \\&\qquad 
 + \frac12 \int \bigg( 
 			2\vA \frac{S^2}{r^3}
 			+ \vA''  \frac{1}{2r}
			 -\vA '  \frac{1}{2r^2}
			- \frac12   \vA'''  
		\bigg)  \phi_{12} ^2,
\end{aligned}
\end{equation*}
and
\begin{equation*}
\begin{aligned}
\dt\wK_{2}
=&\int \left[ \vA  \partial_r  \left( W_{22} \right)+\frac12 \vA '  W_{22}  \right]
 \left( \left( \partial_r -\frac{S}{r}\right)  \phi_{12}-m\phi_{22}\right)
\\&
-\int \left[ \vA  \partial_r \phi_{21}+\frac12 \vA '   \phi_{21}\right]
\left(
		 \left( \partial_r -\frac{S}{r}\right)\left( W_{11} \right)-m W_{21} 
		 \right)
\\&
 +\frac12  \int    \left( 2\vA' -\vA \frac{2}{r} \right)(\partial_r  \phi_{21} )^2
\\&
+\frac12 \int    \bigg(  2\vA \frac{(S+1)^2}{r^3} 
				-  \frac12  \vA '''  
				+ \vA ''   \frac{1}{2r} 
				-\vA '   \frac{1}{2r^2}  
			 \bigg)   \phi_{21}^2
.
\end{aligned}
\end{equation*}
\end{prop}

The proofs of these results follow the same approach. We refer the reader to Appendix \ref{app:Proof_virial} for complete proof of above proposition. We now turn to the complete functional, which enables us to obtain a description of the full spinor \eqref{eq:psi1}. For simplicity, we denote
\be\label{sum}
\K:= \K_1+ \wK_{1} - \K_2- \wK_{2}.
\ee
where the functional are defined in \eqref{eq:K1}, \eqref{eq:tK1}, \eqref{eq:def_K2} and \eqref{eq:def_tK2}.

\begin{rem}[On the sum \eqref{sum}]
We choose \eqref{sum} in this way because it is the only combination that preserves the good sign property of the quadratic terms. This structural feature is crucial to obtain coercive virial-type estimates and to control the leading-order terms. Any other linear combination of the corresponding virial quantities fails to maintain this positivity, which is a key step in the argument.

Although the nonlinear terms play an important role in the full dynamics, they are not dominant in the regime considered here. In this setting, the quadratic (or linearized) part determines the leading-order behavior, ensuring the persistence of the good sign structure. For large solutions, however, the nonlinear coupling could outweigh these favorable contributions, potentially destroying the coercivity obtained for small solutions.
\end{rem}

\begin{cor}\label{corobueno2d}
\new{Under the assumptions of Proposition \ref{prop:K1}}, one has
\begin{equation}\label{derK}
\begin{aligned}
\dt\K
=&~{}
 - \int    \left( \vA' - \frac{\vA}{r} \right)|\nabla \phi|^2
\\&
  -\frac14 \int    \bigg( 
				 4 \frac{(S+1)^2}{r^3} \vA
				 -  \frac{1}{r^2}  \vA ' 
				+ \frac{1}{r}  \vA ''  
				-    \vA '''  
			 \bigg)   (\phi_{21}^2+\phi_{22}  ^2)
 \\&
 - \frac14 \int \bigg( 
 			4 \frac{S^2}{r^3} \vA
			 - \frac{1}{r^2} \vA ' 
 			+  \frac{1}{r}	 \vA''	
			-    \vA'''  
		\bigg) (\phi_{11}^2+ \phi_{12} ^2)
\\&
+m N_1+ N_2+N_3+N_4		
			,
\end{aligned}
\end{equation}
where $N_1$, $N_2$, $N_3$, $N_4$  are the terms related to the nonlinear part, given by
\begin{equation}\label{eq:k1}
\begin{aligned}
N_1=&~{} \int \vA'  \left[  \phi_{11}   W_{11}+\phi_{12}  W_{12}-\phi_{21}  W_{21}-\phi_{22}  W_{22}  \right] 
\\& 
+2 \int  \vA \left[ W_{11}  \partial_r \phi_{11}  +W_{12}  \partial_r \phi_{12}-W_{21}  \partial_r \phi_{21} - W_{22}  \partial_r \phi_{22}\right],
\end{aligned}
\end{equation}

\begin{equation}\label{eq:k2}
\begin{aligned}
& N_2 
= -(2S+1)\int \frac{\vA}{r} \bigg[  
			  \partial_r \phi_{11}  W_{21}
			+  \partial_r \phi_{21} W_{11} 
			+ \partial_r \phi_{12} W_{22}
			+ \partial_r \phi_{22} W_{12} 
	\bigg]
\\&
\quad  -\int \partial_r \left(\frac{\vA}{r}\right) \bigg[  
			S  W_{21} \phi_{11}  
			+ (S+1)  W_{11}  \phi_{21}
			+S  W_{22}  \phi_{12}
			+	(S+1) W_{12}   \phi_{22} 
				\bigg],
\end{aligned}
\end{equation}
\smallskip

\begin{equation}\label{eq:k3}
\begin{aligned}
N_3 =&
~{} \frac12 \int \left(\vA ''  -\frac{\vA '}{r} \right)  \bigg[
					 W_{21} \phi_{11}
					 -  W_{11} \phi_{21} 
					+ W_{22} \phi_{12} 
					- W_{12} \phi_{22} 
			\bigg],
\end{aligned}
\end{equation}
and 
\begin{equation}\label{eq:k4}
\begin{aligned}
N_4=& -2\int \vA \bigg[  
			\partial_r  W_{21} \partial_r  \phi_{11} 
			- \partial_r  W_{11}   \partial_r \phi_{21} 
			+\partial_r W_{22} \partial_r \phi_{12} 	
			- \partial_r W_{12}  \partial_r \phi_{22} 
	\bigg],
\end{aligned}
\end{equation}
are the terms related to the nonlinear part. 
\end{cor}

Notice that the term $N_2$ contains the factor $(2S+1)$ already discovered by Borrelli and Frank \cite{BF_decay}. In the case where $2S+1 =0$, $N_2$ has better decay properties.  

\begin{proof}[Proof of Corollary \ref{corobueno2d}]
Recalling Proposition \ref{prop:K1}, and  after regrouping terms, we obtain
\begin{equation}\label{nueva_J}
\begin{aligned}
\dt \K
=&~{}
 -\frac12  \int    \left( 2\vA' -\vA \frac{2}{r} \right)|\nabla \phi|^2
\\&
  -\frac12 \int    \bigg( 
				 2 \frac{(S+1)^2}{r^3} \vA
				 -  \frac{1}{2r^2}  \vA ' 
				 + \frac{1}{2r}  \vA ''  
				-  \frac12  \vA '''  
			 \bigg)   (\phi_{21}^2+\phi_{22}  ^2)
 \\&
 - \frac12 \int \bigg( 
 			2 \frac{S^2}{r^3} \vA
			 - \frac{1}{2r^2} \vA ' 
 			+  \frac{1}{2r}	 \vA''	
			- \frac12   \vA'''  
		\bigg) (\phi_{11}^2+ \phi_{12} ^2)
\\&+m N_1\new{+N_2+N_3+N_4},
\end{aligned}
\end{equation}
where
\[
\begin{aligned}
m N_1=&
-m\int \vA  \left[  \phi_{11}  \partial_r W_{11}+\phi_{12}  \partial_r W_{12}-\phi_{21}  \partial_r W_{21}-\phi_{22}  \partial_r W_{22}  \right] 
\\& 
+m \int  \vA \left[ W_{11}  \partial_r \phi_{11}  +W_{12}  \partial_r \phi_{12}-W_{21}  \partial_r \phi_{21} - W_{22}  \partial_r \phi_{22}\right] ,
\end{aligned}
\]

\[\begin{aligned}
N_2=&
-\int \frac{\vA}{r} \bigg[  
			-S \partial_r  W_{21} \phi_{11}  
			+ (S+1) \partial_r \phi_{11}  W_{21}
			- (S+1) \partial_r  W_{11}  \phi_{21}
			+S  \partial_r \phi_{21} W_{11} 
	\\&\qquad\quad
			-S \partial_r W_{22}  \phi_{12}
			+ (S+1) \partial_r \phi_{12} W_{22}
				-(S+1) \partial_r W_{12}   \phi_{22} 
				+ S \partial_r \phi_{22} W_{12} 
	\bigg],
\end{aligned}
\]

\[
\begin{aligned}
N_3=&
- \frac12 \int \vA '  \bigg[
					 W_{21} \left( \partial_r -\frac{S}{r}\right) \left(  \phi_{11}  \right)
					- W_{11}  \left( \partial_r +\frac{S+1}{r}\right) \left(  \phi_{21}\right)
		\\&\qquad\qquad\qquad
					+ W_{22}  \left( \partial_r -\frac{S}{r}\right) \left(  \phi_{12}\right)
					- W_{12}  \left( \partial_r +\frac{S+1}{r}\right) \left(  \phi_{22} \right)
			\bigg]
\\&
+\frac12 \int  \vA '  \bigg[  
					   \phi_{21} \left( \partial_r -\frac{S}{r}\right)\left( W_{11} \right)
					+ \phi_{22}  \left( \partial_r -\frac{S}{r}\right) \left( W_{12} \right)
			\\&\qquad\qquad\qquad
					- \phi_{11} \left( \partial_r +\frac{S+1}{r}\right)  \left(W_{21}\right)
					-  \phi_{12}   \left( \partial_r +\frac{S+1}{r}\right) \left(W_{22}\right)
			\bigg],
\end{aligned}
\]
and 
\begin{equation*}
\begin{aligned}
N_4=& -2\int \vA \bigg[  
			\partial_r  W_{21} \partial_r  \phi_{11} 
			- \partial_r  W_{11}   \partial_r \phi_{21} 
			+\partial_r W_{22} \partial_r \phi_{12} 	
			- \partial_r W_{12}  \partial_r \phi_{22} 
	\bigg].
\end{aligned}
\end{equation*}

Let us focus on $N_1$. Notice that the terms where the mass $m$ is involved, after some integration by parts, satisfy
\begin{equation*}
\begin{aligned}
mN_1=& -m\int \vA  \left[  \phi_{11}  \partial_r W_{11}+\phi_{12}  \partial_r W_{12}-\phi_{21}  \partial_r W_{21}-\phi_{22}  \partial_r W_{22}  \right] 
\\& 
+m \int  \vA \left[ W_{11}  \partial_r \phi_{11}  +W_{12}  \partial_r \phi_{12}-W_{21}  \partial_r \phi_{21} - W_{22}  \partial_r \phi_{22}\right] 
\\
=&~{} m\int \vA'  \left[  \phi_{11}   W_{11}+\phi_{12}  W_{12}-\phi_{21}  W_{21}-\phi_{22}  W_{22}  \right] 
\\& 
+2m \int  \vA \left[ W_{11}  \partial_r \phi_{11}  +W_{12}  \partial_r \phi_{12}-W_{21}  \partial_r \phi_{21} - W_{22}  \partial_r \phi_{22}\right] .
\end{aligned}
\end{equation*}
Now, let us focus on one of the terms that involves the \new{vorticity} $S$, rearranging the terms,
\[
\begin{aligned}
N_2=&
-\int \frac{\vA}{r} \bigg[  
			-S \partial_r  W_{21} \phi_{11}  
			+ (S+1) \partial_r \phi_{11}  W_{21}
			- (S+1) \partial_r  W_{11}  \phi_{21}
			+S  \partial_r \phi_{21} W_{11} 
	\\&\qquad\qquad
			-S \partial_r W_{22}  \phi_{12}
			+ (S+1) \partial_r \phi_{12} W_{22}
				-(S+1) \partial_r W_{12}   \phi_{22} 
				+ S \partial_r \phi_{22} W_{12} 
	\bigg],
\end{aligned}
\]	
and integrating by parts and using that \new{$\varphi(0)=\varphi'(0)=0$ (which also gives $\displaystyle \lim_{r\to 0}\varphi(r)/r=0$)}, we have
\begin{equation*}
\begin{aligned}
N_2
=&  -(2S+1)\int \frac{\vA}{r} \bigg[  
			  \partial_r \phi_{11}  W_{21}
			+  \partial_r \phi_{21} W_{11} 
			+ \partial_r \phi_{12} W_{22}
			+ \partial_r \phi_{22} W_{12} 
	\bigg]
\\&
 -\int \partial_r \left(\frac{\vA}{r}\right) \bigg[  
			S  W_{21} \phi_{11}  
			+ (S+1)  W_{11}  \phi_{21}
			+S  W_{22}  \phi_{12}
			+	(S+1) W_{12}   \phi_{22} 
	\bigg].
\end{aligned}
\end{equation*}
For the last two terms, rearranging the terms, we have
\[
\begin{aligned}
N_3
=&
- \frac12 \int \vA '  \bigg[
					 W_{21} \left( \partial_r -\frac{S}{r}\right) \left(  \phi_{11}  \right)
					 + \phi_{11} \left( \partial_r +\frac{S+1}{r}\right) W_{21}
		\\&\qquad\qquad	
					- W_{11}  \left( \partial_r +\frac{S+1}{r}\right) \left(  \phi_{21}\right)
					 - \phi_{21} \left( \partial_r -\frac{S}{r}\right)W_{11} 
		\\&\qquad\qquad
					+ W_{22}  \left( \partial_r -\frac{S}{r}\right) \left(  \phi_{12}\right)
					+ \phi_{12}   \left( \partial_r +\frac{S+1}{r}\right) W_{22}
		\\&\qquad\qquad	
					- W_{12}  \left( \partial_r +\frac{S+1}{r}\right) \left(  \phi_{22} \right)
					- \phi_{22}  \left( \partial_r -\frac{S}{r}\right)W_{12} 
			\bigg]
\\
=&
- \frac12 \int \vA '  \bigg[
					\partial_r ( W_{21} \phi_{11})
					 + \frac{1}{r} W_{21} \phi_{11}
					- \frac{1}{r} W_{11} \phi_{21} 
					 -  \partial_r \left( W_{11} \phi_{21} \right)
		\\&\qquad\qquad
					+\frac{1}{r} W_{22}   \phi_{12}
					+  \partial_r  \left(W_{22} \phi_{12} \right)
					-\frac{1}{r} \left( W_{12}  \phi_{22} \right)
					-   \partial_r  \left( W_{12} \phi_{22} \right)
			\bigg].
\end{aligned}
\]
Then, we obtain that 
\begin{equation*}
\begin{aligned}
N_3
=&
~{} \frac12 \int \left(\vA ''  -\frac{\vA '}{r} \right)  \bigg[
					 W_{21} \phi_{11}
					 -  W_{11} \phi_{21} 
					+ W_{22} \phi_{12} 
					- W_{12} \phi_{22} 
			\bigg].
\end{aligned}
\end{equation*}
Finally, the proof of this proposition ends by collecting the previous identities.
\end{proof}

\section{High-order nonlinearities: bounds on the \new{quadratic and} nonlinear terms}\label{sec:thm+}

Let us start the proof of Theorem \ref{thm:2D_m+}. In this section, we consider the massless case ($m=0$), which implies that $N_1$ in \eqref{derK} is not considered. Our analysis, therefore, focuses on the remaining terms $N_2$, $N_3$, and $N_4$.

\subsection{\new{Nonlinear term in the} massless case} 

Let us focus on the terms $N_2$, $N_3$, and $N_4$ in the Corollary \ref{corobueno2d} (see \eqref{eq:k2},\eqref{eq:k3} and \eqref{eq:k4}, respectively). Using \eqref{eq:W_2D} in \eqref{deco_complex}, and \eqref{mod1}-\eqref{mod2},
\[
\left|\partial_r \phi_{11}  W_{21}
			+  \partial_r \phi_{21} W_{11} 
			+ \partial_r \phi_{12} W_{22}
			+ \partial_r \phi_{22} W_{12} \right| \lesssim  | \nabla \phi| |\phi|^p.
\]
Similarly,
\[
\left| S  W_{21} \phi_{11}  
			+ (S+1)  W_{11}  \phi_{21}
			+S  W_{22}  \phi_{12}
			+	(S+1) W_{12}   \phi_{22}  \right| 
            \new{\lesssim} |\phi|^{p+1}.
\]
\new{Here and below all implicit constants in $\lesssim$ are allowed to depend on the fixed vorticity $S\neq0,-1$.} 

Firstly, for $N_2$, using \new{the Cauchy--Schwarz inequality} along with the previous estimates we have
\[
\begin{aligned}
|N_2|
=&~{}  |2S+1|\left|\int \frac{\vA}{r} \bigg[  
			  \partial_r \phi_{11}  W_{21}
			+  \partial_r \phi_{21} W_{11} 
			+ \partial_r \phi_{12} W_{22}
			+ \partial_r \phi_{22} W_{12} 
	\bigg]\right|
\\&
 +\left|\int \partial_r \left(\frac{\vA}{r}\right) \bigg[  
			S  W_{21} \phi_{11}  
			+ (S+1)  W_{11}  \phi_{21}
			+S  W_{22}  \phi_{12}
			+	(S+1) W_{12}   \phi_{22} 
	\bigg]\right|
	\\
\lesssim &~{}  |2S+1| \int \frac{\vA}{r} | \nabla \phi| |\phi|^p  
  + \int \left| \partial_r \left(\frac{\vA}{r}\right)\right| |\phi|^{p+1}
  	\\
\lesssim &~{}\int \frac{\vA}{r} ( | \nabla \phi|^2|\phi|^2+ |\phi|^{2+2(p-2)} ) 
  + \int \left| \partial_r \left(\frac{\vA}{r}\right)\right| |\phi|^{p+1}.
 \end{aligned}
\]

\new{As stated above, the implicit constants below may depend on the fixed $S\neq0,-1$; we do not track this dependence explicitly.} By a direct application of Lemma \ref{lem2d}, we have
\[
\begin{aligned}
|N_2|\lesssim &~{}  |2S+1| \int \frac{\vA}{r} \left(\frac{\epsilon^2 | \nabla \phi|^2 }{(1+r)}+\frac{\epsilon^{2(p-2)} |\phi|^2}{(1+r)^{p-2}} \right)
\\&~\qquad
  +(\new{|S|+1})\int \frac{\epsilon^{p-1}}{(1+r)^{(p-1)/2}}\left| \partial_r \left(\frac{\vA}{r}\right)\right| |\phi|^2.
\end{aligned}
\]

For the second term, we have 
\begin{equation}\label{N3}
\begin{aligned}
|N_3| =&
~{} \frac12 \left| \int \left(\vA ''  -\frac{\vA '}{r} \right)  \bigg[
					 W_{21} \phi_{11}
					 -  W_{11} \phi_{21} 
					+ W_{22} \phi_{12} 
					- W_{12} \phi_{22} 
			\bigg]\right|
\\
\leq&
~{} \frac12 \int \left|\vA ''  -\frac{\vA '}{r} \right|  |\phi|^{p+1}
\leq ~{} C \epsilon^{p-1}  \int \left|\vA ''  -\frac{\vA '}{r} \right| \frac{ |\phi|^2}{(1+r)^{(p-1)/2}} .
\end{aligned}
\end{equation}
Finally, the last term is bounded as follows:
\begin{equation}\label{N4}
\begin{aligned}
|N_4| =&~{} 2\left|\int \vA \bigg[  
			\partial_r  W_{21} \partial_r  \phi_{11} 
			- \partial_r  W_{11}   \partial_r \phi_{21} 
			+\partial_r W_{22} \partial_r \phi_{12} 	
			- \partial_r W_{12}  \partial_r \phi_{22} 
	\bigg]\right|
	\\
	\leq&~{} C \int \vA |\nabla \phi|^2 |\phi|^{p-1}
	\leq C\epsilon^{p-1} \int \frac{\vA  }{(1+r)^{(p-1)/2}} |\nabla \phi|^2.
\end{aligned}
\end{equation}
Then, we conclude
\[
\begin{aligned}
|N_2+N_3+N_4|
\leq&~{}
|2S+1| \int \frac{\vA}{r} \frac{\epsilon^2  | \nabla \phi|^2}{(1+r)}
+C\epsilon^{p-1} \int \frac{\vA |\nabla \phi|^2 }{(1+r)^{(p-1)/2}} 
\\&
+|2S+1| \int \frac{\vA}{r}\frac{\epsilon^{2(p-2)}}{(1+r)^{p-2}}  |\phi|^2
\\&
  +(\new{|S|+1})\int\frac{\epsilon^{p-1}}{(1+r)^{(p-1)/2}}\left| \partial_r \left(\frac{\vA}{r}\right)\right| |\phi|^2
\\&+C \epsilon^{p-1}  \int \frac{1}{(1+r)^{(p-1)/2}}\left|\vA ''  -\frac{\vA '}{r} \right|  |\phi|^2
.
\end{aligned}
\]

Finally,  since $p\geq 5$ (see \ref{thm:+_massless_Strong} in Theorem \ref{thm:2D_m+}), we conclude that
\[
\begin{aligned}
|N_2+N_3+N_4|
\lesssim&~{}
\epsilon^2  \int \frac{\vA}{r(1+r)} | \nabla \phi|^2
+\epsilon^4 \int \frac{\vA}{r (1+r)^{3}} |\phi|^2
\\&
+\epsilon^4 \int \left[\frac{1}{(1+r)^{2}}\left| \frac{\vA'}{r}-\frac{\vA}{r^2} \right| 
+ \frac{1}{(1+r)^{2}}\left|\vA ''  -\frac{\vA '}{r} \right|  \right]|\phi|^2
.
\end{aligned}
\]
Now we choose the particular weight function. This choice is motivated by the strength of the quadratic terms, the order of the nonlinearities, and the numerous nonlinear terms. 

\begin{lem}\label{Le1}
Let us consider 
\begin{equation}\label{eq:vA_HN}
\vA  = \frac{r^{3}}{(1+r)^2}. 
\end{equation}
Then $\vA $ is a smooth bounded function such that $\vA (0)=\vA' (0)=0$, and 
\begin{equation}\label{eq:va_cub}
\begin{aligned}
 \vA '=~{} \frac{r^2(r+3)}{(1+r)^3} &,~{} \new{\vA''= \frac{6 r}{(1 + r)^4}, ~{}\vA'''=\frac{6 - 18 r}{(1 + r)^5} }\\
 \vA '- \frac{\vA  }{r}  = \frac{2r^2}{(1+r)^3}, & \quad   \frac{\vA  }{r(1+r)}  =  \frac{r^2}{(1+r)^3},
\\
\vA ''- \frac{\vA'  }{r}  =&-\frac{r^3+4r^2-3r}{(r+1)^4}
.
 \end{aligned}
 \end{equation}
\end{lem}

The proof of this result is direct. Therefore, using Lemma \ref{Le1}, we obtain that 
\begin{equation}\label{eq:N2+N3+N4}\color{blue}
\begin{aligned}
|N_2+N_3+N_4|
\lesssim&~{}
\epsilon^2  \int \frac{r^2}{(1+r)^3} | \nabla \phi|^2
\\&
+\epsilon^4 \int \left[\frac{r^2}{ (1+r)^{5}}  + \frac{2r}{(1+r)^{6}} 
+ \frac{r^3+4r^2+3r}{(1+r)^{6}} \right]|\phi|^2
\\
\lesssim&~{}
\epsilon^2  \int \frac{r^2}{(1+r)^3} | \nabla \phi|^2
+\epsilon^4 \int  \frac{1}{(1+r)^{3}} |\phi|^2
.
\end{aligned}
\end{equation}

\subsection{Quadratic terms \new{in the massive case}}
Now, let us focus on the weight functions for $|\phi_1|^2$ and $|\phi_2|^2$. For $K\neq0$, using \eqref{eq:vA_HN}, we get
 \begin{equation}\label{quadratic_part}
\begin{aligned}
 4 K^2 \frac{\vA }{r^3}
 			 - \frac{1}{r^2}\vA ' 
			 &
 			 + \frac{1}{r}\vA'' 
			-    \vA'''   
			\\
			=&~{}
			\frac{1}{(1+r)^5}\big( 4 K^2 (r+1)^3- r^3-5r^2+17r-3 \big)
\\
			=&~{}
			\frac{1}{(1+r)^5}\big( (4 K^2-1) (r+1)^3-2(r+1)^2+24r \big).
\end{aligned}
\end{equation}
We remark that 
\new{the above quantity is positive for every $K\in\mathbb{Z}\setminus\{0\}$; in particular it is positive in the two cases $K=S$ and $K=S+1$ relevant here, since $S\neq0,-1$.} 
Finally,  recalling \eqref{derK} and using the above identities and \eqref{eq:N2+N3+N4}, we obtain that 
 \begin{equation*}
\begin{aligned}
\dt \K \leq&~{}
 -(1-C\epsilon^2) \int    \frac{r^2}{(1+r)^3}|\nabla \phi|^2
\\&
  -\frac12 \int    
  \frac{1}{2(1+r)^5} M_1(S+1,r,\epsilon)
			    (\phi_{21}^2+\phi_{22}  ^2)
 \\&
 - \frac12 \int 
 \frac{1}{2(1+r)^5}M_1(S,r,\epsilon)
		 (\phi_{11}^2+ \phi_{12} ^2),
\end{aligned}
\end{equation*}
with
\[
\ba
M_1(S,r,\epsilon):= &~{} (4 S^2-1) (r+1)^3-2(r+1)^2+24r
				-4C\epsilon^4 (r+1)^2.
\ea
\]

Then, we conclude that
 \begin{equation*}
\begin{aligned}
-\dt \K
\gtrsim&~{}
 \int    \frac{r^2}{(1+r)^3}|\nabla \phi|^2
  +\new{ \int \frac{r^3+r^2+r+1}{(1+r)^5} |\phi|^2}. 
\end{aligned}
\end{equation*}
\new{(Here we use that $\varepsilon$ is small, so that $1-C\varepsilon^2>0$ and $M_1(S,r,\varepsilon),\,M_1(S+1,r,\varepsilon)\geq C (1+r)^3>0$ for some $C>0$, which yields the coercive lower bound.)}

\subsection{\new{Nonlinear terms in the } massive case}\label{sub3p3}

Now, let us deal with the term appearing in Corollary \ref{corobueno2d}, which involves the mass term, i.e., the quantity $N_1$ obtained in \eqref{eq:k1}. Using \eqref{eq:W_2D} and Cauchy-Schwarz inequality, then $N_1$ is bounded by
\be\label{cota_N1}
\begin{aligned}
|N_1|
\leq&~{}
\bigg|\int \vA'  \left[  \phi_{11}   W_{11}+\phi_{12}  W_{12}-\phi_{21}  W_{21}-\phi_{22}  W_{22}  \right] \bigg|
\\& 
+2\bigg| \int  \vA \left[ W_{11}  \partial_r \phi_{11}  +W_{12}  \partial_r \phi_{12}-W_{21}  \partial_r \phi_{21} - W_{22}  \partial_r \phi_{22}\right] \bigg|
\\
\lesssim&~{}
\int |\vA'| |\phi|^{p+1}
+\int  \vA |\nabla \phi| |\phi|^{p}
\\
\lesssim&~{}
\int |\vA'| |\phi|^{p+1}
+\int  \vA (|\nabla \phi|^2|\phi|^4+ |\phi|^{2+2(p-3)}).
\end{aligned}
\ee
Applying Lemma \ref{lem2d}, we obtain
\[
\begin{aligned}
|N_1|
\lesssim&~{}
C \epsilon^4\int  \frac{\vA |\nabla \phi|^2}{(1+r)^2}
+C\epsilon^{p-1}\int \frac{|\vA'| |\phi|^{2} }{(1+r)^{(p-1)/2}} 
+C\epsilon^{2(p-3)}\int  \frac{\vA  |\phi|^{2}}{(1+r)^{p-3}} 
.
\end{aligned}
\]
By recalling the case \ref{thm:+_massive} and the form of $\vA$ from \eqref{eq:vA_HN}, and then using the fact that $p\geq 7$, we obtain
\[
\begin{aligned}
|N_1|
\leq&~{}
C \epsilon^4\int  \frac{\vA }{(1+r)^2}|\nabla \phi|^2
+C\epsilon^{6}\int \bigg( \frac{|\vA'|}{(1+r)^{3}}+ \frac{\vA}{(1+r)^{4}} \bigg) |\phi|^{2}
\\
\leq&~{}
C \epsilon^4\int  \frac{r^3}{(1+r)^4}|\nabla \phi|^2
+C\epsilon^{6}\int \bigg( \frac{r^2(r+3)}{(1+r)^{6}}+ \frac{r^3}{(1+r)^{6}} \bigg) |\phi|^{2}
\\
\leq&~{}
C \epsilon^4\int  \frac{r^2}{(1+r)^3}|\nabla \phi|^2
+C\epsilon^{6}\int \frac{r^3+r^2}{(1+r)^{6}}  |\phi|^{2}
.
\end{aligned}
\]
 \new{Similar to the massless case}, we conclude that
 \begin{equation}\label{final_K+}
\begin{aligned}
-\dt \K
\gtrsim &~{}
 \int    \frac{r^2}{(1+r)^3}|\nabla \phi|^2
  +\frac12 \int
  \frac{r^3+r^2+r+1}{(1+r)^5} |\phi|^2 \\
  \gtrsim &~{}
 \int    \frac{r^2}{(1+r)^3}|\nabla \phi|^2
  + \int
  \new{\frac{1}{(1+r)^{2}}} |\phi|^2.
\end{aligned}
\end{equation}

This estimate will be key to prove the Theorem \ref{thm:2D_m+}.

\begin{rem}\label{3p1}
It is interesting to compare the previous result (the bound on $N_1$ from \eqref{cota_N1}) to certain explicit regimes. In the Borrelli-Frank regime, the nonlinearity is $W=|\psi|^2\psi$, and
\[
W_{11}  \partial_r \phi_{11}  +W_{12}  \partial_r \phi_{12}-W_{21}  \partial_r \phi_{21} - W_{22}  \partial_r \phi_{22}= \frac12|\psi|^2 (|\psi_1|^2-|\psi_2|^2)_r.
\]
In Soler's case, 
$W=
\new{\langle\psi,\beta\psi\rangle} \beta \psi,$ and
\[
W_{11}  \partial_r \phi_{11}  +W_{12}  \partial_r \phi_{12}-W_{21}  \partial_r \phi_{21} - W_{22}  \partial_r \phi_{22}= \frac12\new{(|\psi_1|^2-|\psi_2|^2) (|\psi|^2)_r .}
\]

In the Honeycomb case, $W_1=(\new{\beta_1} |\psi_1|^2+\new{2\beta_2}|\psi_2|^2) \psi_1$, $W_2 =(\new{2\beta_2} |\psi_1|^2+\new{\beta_1}|\psi_2|^2) \psi_2$, and 
\[
\ba
&W_{11}  \partial_r \phi_{11}  +W_{12}  \partial_r \phi_{12}-W_{21}  \partial_r \phi_{21} - W_{22}  \partial_r \phi_{22} \\
&= \frac12 \new{\beta_1}\left(  |\phi_1|^2(|\phi_1|^2)_r-|\phi_2|^2(|\phi_2|^2)_r \right) +\frac12 \new{\beta_2} \left(  |\phi_2|^2(|\phi_1|^2)_r-|\phi_1|^2(|\phi_2|^2)_r \right) \\
&= \frac14 \new{\beta_1} ( |\phi_1|^4-  |\phi_2|^4)_r +\frac12 \new{\beta_2} \left(  |\phi_2|^2(|\phi_1|^2)_r-|\phi_1|^2(|\phi_2|^2)_r \right).
\ea
\]

\new{
Consequently, the ``bad behaviour'' of $N_1$ means that it cannot be 
written as a total derivative, or as a divergence, and therefore contributes 
a genuine non-conservative term to the virial identity. It is worth noting 
that for some specific nonlinearities, such as the pseudo-scalar model 
$W=\langle\psi,\beta\psi\rangle\psi$, the term $N_1$ can indeed be 
rewritten as a complete derivative. Nevertheless, for the sake of 
completeness, we have decided to maintain the general hypotheses. 
}

\new{
The 
hypotheses on the nonlinearity required in the main result are therefore 
\emph{natural}, in the sense that they are forced by the structure of 
any non-trivially coupled two-component power-type nonlinearity. We do 
not claim here whether these conditions are \emph{optimal}; the only 
regimes where a total derivative appears are the decoupled ones, such 
as $\beta_2=0$ in the Honeycomb case.
}

\end{rem}

\subsection{Bound on the virial functionals}\label{3p4}
Now we obtain precise bounds on the functional $\K_1$. A bound on $\K_2$ and the remaining functionals is \new{analogous}. 

From \eqref{eq:K1},
\[
\begin{aligned}
\left| \K_{1}\right| \lesssim &~{} \int  \vA |\partial_r \phi_{11}| |\partial_r  \phi_{22}| + \int  \frac1r\vA |\partial_r \phi_{11}|  |\phi_{22}| + \int  \vA |\partial_r \phi_{11}|  |\phi_{12}| \\
&~{} +\int |\vA '| |\phi_{11}| | \partial_r  \phi_{22}|  +\int  \frac1r |\vA '| |\phi_{11}| |   \phi_{22}|  +\int |\vA '| |\phi_{11}|  |\phi_{12}|.
\end{aligned}
\]
Recalling \eqref{eq:va_cub} and focusing on the interval $r\leq 1$, we have
\[
\begin{aligned}
|\vA'|+\frac{1}{r}|\vA'|+|\vA|+\frac{1}{r}|\vA|
=&~{}\frac{r(r+1)(r+3)}{(1+r)^3} +\frac{r(r^{2}+r)}{(1+r)^2}
\lesssim r
.
\end{aligned}
\]
Similarly, for $r>1$, one \new{observes}
\[
\begin{aligned}
|\vA'|+\frac{1}{r}|\vA'|+|\vA|+\frac{1}{r}|\vA|
\leq &~{}\frac{3r(1+r)^2}{(1+r)^3} +\frac{r(1+r)^2}{(1+r)^2}
\lesssim r
.
\end{aligned}
\]
Finally, recalling the notation in \eqref{mod1}-\eqref{mod2} and using Cauchy-Schwarz inequality, we conclude
\[
\begin{aligned}
\left| \K_{1}\right| \lesssim &~{} \int  ( |\nabla \phi|^2+ |\phi|^2)r dr.
\end{aligned}
\]
Using \new{the fact that $\sup_{t\geq 0} \|{\bf \phi}(t)\|_{H^1} <\varepsilon$ (recall that $\|\phi\|_{H^1}^2=\displaystyle \int(|\nabla\phi|^2+|\phi|^2)\,r\,dr$)}, we conclude the uniform bound on $\K_{1}$.

\section{Low-order nonlinearities: bounds on the \new{quadratic part and} nonlinear terms}\label{sec:thm-}

In this section, let us focus on terms that do not involve the mass, i.e., $N_2$, $N_3$, and $N_4$. These bounds will be useful for the proof of Theorem \ref{thm:2D_m-}.

\subsection{Massless case}

Let us focus on the terms $N_2$, $N_3$, and $N_4$ in the Corollary \ref{corobueno2d} (see \eqref{eq:k2},\eqref{eq:k3} and \eqref{eq:k4}, respectively).

Firstly, \new{we emphasize that the treatment of the term $N_2$ differs slightly in this setting. Although this is motivated by the space $E(\delta)$ and the procedure is analogous, we include it for the sake of clarity.} For $N_2$, using Cauchy-Schwarz along with \eqref{eq:W_2D}, we have
\[
\begin{aligned}
|N_2|
\leq&~{} |2S+1| \int \frac{\vA}{r} | \nabla \phi| |\phi|^p  
  +(\new{|S|+1})\int\left| \partial_r \left(\frac{\vA}{r}\right)\right| |\phi|^{p+1}
  	\\
\leq&~{}  |2S+1| \int \frac{\vA}{r} ( \epsilon| \nabla \phi|^2+ \epsilon^{-1}|\phi|^{2+2(p-1)} ) 
  +(\new{|S|+1})\int\left| \partial_r \left(\frac{\vA}{r}\right)\right| |\phi|^{p+1};
 \end{aligned}
\]
 and by a direct application of Lemma \ref{lem2d}, we obtain
\[
\begin{aligned}
|N_2| \leq&~{}  |2S+1| \int \frac{\vA}{r} \left( \epsilon | \nabla \phi|^2+ |\phi|^2\frac{\epsilon^{2p-3}}{(1+r)^{p-1}} \right) 
\\&
  +(\new{|S|+1})\int\frac{\epsilon^{p-1}}{(1+r)^{(p-1)/2}}\left| \partial_r \left(\frac{\vA}{r}\right)\right| |\phi|^2
      	\\
\lesssim&~{}  \epsilon \int \frac{\vA}{r} | \nabla \phi|^2
+  \epsilon^{2p-3} \int \frac{\vA}{r(1+r)^{p-1}} |\phi|^2
  +\epsilon^{p-1} \int \frac{1}{(1+r)^{(p-1)/2}}\left| \partial_r \left(\frac{\vA}{r}\right)\right| |\phi|^2
  .
\end{aligned}
\]

\new{Recalling the bounds in \eqref{N3} and \eqref{N4}, along with the} hypothesis \emph{\ref{thm:-_m-_w}},  we have  $p\geq 3$, and therefore
\[
\begin{aligned}
|N_2+N_3+N_4|
\lesssim&~{}
 \epsilon \int \frac{\vA}{r} | \nabla \phi|^2
+\epsilon^{2} \int \frac{\vA}{(1+r)} |\nabla \phi|^2 
\\&
+  \epsilon^{3} \int \frac{\vA}{r(1+r)^{2}} |\phi|^2
  +\epsilon^{2} \int \frac{1}{(1+r)}\left| \partial_r \left(\frac{\vA}{r}\right)\right| |\phi|^2
\\&
+  \epsilon^{2}  \int \frac{1}{(1+r)}\left|\vA ''  -\frac{\vA '}{r} \right|  |\phi|^2
.
\end{aligned}
\]
Now, we shall use the following result.

\begin{lem}\label{Le2}
Let
\begin{equation}\label{eq:varphi}
\vA (r) =\frac{r^{3+\de}}{(1+r)^2}, \quad \de>0.
\end{equation}
Then, we have
\begin{equation}\label{eq:va_delta}
\begin{aligned}
& \vA '= \frac{r^{2+\de}(3+\de+(1+\de) r)}{(1+r)^3} ,\\
& \vA '- \frac{\vA  }{r}  = \frac{r^{2+\de}(2+\de+\de r)}{(1+r)^3},  \quad   \frac{\vA  }{r}  =  \frac{r^{2+\de}}{(1+r)^2},
\\
& \vA ''- \frac{\vA'  }{r}  =~{}  \frac{r^{\de+1}}{(r+1)^4} ((-1+\de^2) r^2+2( \de^2+2\de-2)r+3+\de^2 +4\de  )
.
 \end{aligned}
 \end{equation}
 \end{lem}
 The proof of Lemma \ref{Le2} is direct. The choice of $\varphi$ is motivated by the particularity of the 2d case and the lower value of the power of the nonlinearity, and it will lead to a better control on the decay estimates.

Computing for a fixed $K\in \bb{Z}$, and using \eqref{eq:va_delta}, one obtains
\[
\begin{aligned}
2& \frac{K^2}{r^3} \vA
			 - \frac{1}{2r^2} \vA ' 
 			+  \frac{1}{2r}	 \vA''	
			- \frac12   \vA''' 
			\\ 
			=&~{}  (4 K^2-\de^3)  \frac{r^{\de}}{2 (r+1)^5}  (r+1)^3
			\\&
			     + \frac{r^{\de}}{2 (r+1)^5} \left(\de^2(r^3-3 r^2-9 r-5)+\de (r^3+13 r^2+5 r-7)-  r^3-5r^2+17r -3\right)
			     .
\end{aligned}
\]			     
\new{After rearranging the terms, we conclude}
\[
\begin{aligned}
2 \frac{K^2}{r^3} \vA&
			 - \frac{1}{2r^2} \vA ' 
 			+  \frac{1}{2r}	 \vA''	
			- \frac12   \vA''' 
			\\ 			     
			=&~{}  (4 K^2-1-\de^3+\de^2+\de)  \frac{r^{\de}}{2 (r+1)^5}  (r+1)^3
			   \\
			   &  -2(1-5\de+3\de^2) \frac{r^{\de}}{2 (r+1)^5}(r+1)^2 
			\\&
			+(24r-18\de(1+r))  \frac{r^{\de}}{2 (r+1)^5} 	
			.
\end{aligned}
\]
Finally, recalling that $S\neq0,-1$ along with the previous computation using the weight function \eqref{eq:varphi}, after some standard estimates, we have
\begin{equation*}
\begin{aligned}
\dt \K
\leq&~{}
 -  \int   \frac{r^{2+\de}(2+\de+\de r)}{(1+r)^3}|\nabla \phi|^2
  -\frac{1}{8} \int   
			  \frac{r^{\de}}{(r+1)^2}|\phi|^2
\\&
+C\epsilon \int \frac{r^{2+\de}}{(1+r)^2} |\nabla \phi|^2 
\\&
+ C \epsilon^{2} \int
	\frac{r^{\de}}{(1+r)^5}\bigg(
							(2+\de+\de^2)r^3
							+(6+6\de+2\de^2)r^2
							+ (\de^2 +5\de +5) r
				 \bigg)|\phi|^2
.
\end{aligned}
\end{equation*}
Therefore, since $\epsilon$ is small,  we conclude
\begin{equation}\label{final_K}
\begin{aligned}
-\dt \K
\gtrsim&~{}
   \int   \frac{r^{2+\de}}{(1+r)^2}|\nabla \phi|^2
+ \int     \frac{r^{\de}}{  (r+1)^2}|\phi|^2
.
\end{aligned}
\end{equation}

\subsection{Massive case}

Notice that for the \emph{massive case}, we deal with the nonlinearities involved in the term $N_1$. \new{Working in the space $E(\delta)$ allows us to split the terms in a more convenient way to reach a lower-order nonlinearity (compare the following inequality with \eqref{cota_N1}). We have}
\[\begin{aligned}
|N_1|
\lesssim&~{}
\int |\vA'| |\phi|^{p+1}
+\int  \vA |\nabla \phi| |\phi|^{p}
\\
\lesssim&~{}
\epsilon^{p-1}\int \frac{|\vA'|}{(1+r)^{(p-1)/2}} |\phi|^{2}
+\int  \vA (|\nabla \phi|^2|\phi|^2+ |\phi|^{2+2(p-2)})
\\
\lesssim&~{}
C \epsilon^2\int  \frac{\vA }{(1+r)}|\nabla \phi|^2
+C\epsilon^{p-1}\int \frac{|\vA'|}{(1+r)^{(p-1)/2}} |\phi|^{2}
\\&
+C\epsilon^{2(p-2)}\int  \frac{\vA}{(1+r)^{p-2}}  |\phi|^{2}
.
\end{aligned}
\]
Then, since $p\geq p_*=5$, along with \eqref{eq:varphi}, we obtain
\[\begin{aligned}
|N_1|
\lesssim&~{}
C \epsilon^2\int  \frac{r^{3+\de}}{(1+r)^3}|\nabla \phi|^2
+C\epsilon^{6}\int  \frac{r^{3+\de}}{(1+r)^{5}}  |\phi|^{2}
\\&
+C\epsilon^{4}\int \frac{r^{2+\de}}{(1+r)^5} (3+\de+(1+\de)r)|\phi|^{2}
\\
\lesssim&~{}
C \epsilon^2\int  \frac{r^{2+\de}}{(1+r)^2}|\nabla \phi|^2
+C\epsilon^{4}\int  \frac{r^{\de}}{(1+r)^{2}}  |\phi|^{2}
.
\end{aligned}
\]
By arguing as in the massless case, we obtain \eqref{final_K}.

\subsection{A well-defined functional}
Recalling \eqref{eq:va_cub} and focusing on the interval $r\leq 1$, we have
\[
\begin{aligned}
|\vA'|+\frac{1}{r}|\vA'|+|\vA|+\frac{1}{r}|\vA|
=&~{}\frac{r^{1+\de}(1+r)(3+\de+(1+\de) r)}{(1+r)^3}+\frac{r^{1+\de}(r^2+r)}{(1+r)^2}
\\
\lesssim&~{}  r^{1+\de}.
\end{aligned}
\]
Similarly, for $r>1$, one observes
\[
\begin{aligned}
|\vA'|+\frac{1}{r}|\vA'|+|\vA|+\frac{1}{r}|\vA|
\leq &~{}(3+\de)\frac{r^{1+\de}(1+r)^2}{(1+r)^3} +\frac{r^{1+\de}(1+r)^2}{(1+r)^2}
\lesssim r^{1+\de}
.
\end{aligned}
\]
Finally, recalling the notation in \eqref{mod1}-\eqref{mod2} and using Cauchy-Schwarz inequality, we conclude
\[
\begin{aligned}
\left| \K_{1}\right| \lesssim &~{} \int  ( |\nabla \phi|^2+ |\phi|^2)r^{1+\de} dr.
\end{aligned}
\]
Again, using \eqref{eq:weight} and the fact that $\sup_{t\geq 0} \|{\bf \phi}(t)\|_{E(\delta)} <\varepsilon$, we conclude the uniform bound on $\K_{1}$.

\section{End of proof of Theorems \ref{thm:2D_m-} and \ref{thm:2D_m+}}\label{Section:2D} 

The following argument applies to both \eqref{limit2D_m-} in Theorem \ref{thm:2D_m-} and \eqref{limit2D} in Theorem \ref{thm:2D_m+}, using the corresponding bounds \eqref{final_K} and \eqref{final_K+}, respectively. We focus here on the proof of Theorem \ref{thm:2D_m-}.

\begin{lem}
Let $\phi$ be a solution to \eqref{eq:2Dr}, and consider the functional
\be\label{mathH}
\mathcal{H}(t) = \int_0^\infty \vA  |\phi|^2.
\ee
Assume that $\vA$ is smooth, $\vA(0)=0$, $\vA(r) \to 0$ as $r\to +\infty.$ Then one has
\be\label{dH}
\dfrac{d}{dt}\mathcal{H}(t)  = -2\Im \int_0^\infty \left( \vA'  - \frac{1}{r}\vA \right) \overline{\phi_1}\phi_2 
+2 \Im \int \vA  (\overline{\phi_2} W_2+\overline{\phi_1} W_1).
\ee
\end{lem}

\begin{proof}
We can see that 
\[ \begin{aligned}
\dfrac{d}{dt}\mathcal{H}(t) 
=&~{}2\Re \int_0^\infty \vA (-i)(i\partial_t \phi_1 \overline{\phi_1}+i\partial_t \phi_2 \overline{\phi_2})
\\
=&~{}2\Im \int_0^\infty \vA   \left( 
							\overline{\phi_1}  \partial_r  \phi_2 
							-\overline{\phi_2}\partial_r \phi_1 
						 \right)
	\\&
+2\Im \int_0^\infty \vA   \left( 
							   \frac{1}{r} \overline{\phi_1} \phi_2 
							+  \frac{S}{r}( \overline{\phi_2}\phi_1 +\overline{\phi_1} \phi_2 )
						 \right)
\\&
+2 \Im \int \vA  (\overline{\phi_2} W_2+\overline{\phi_1} W_1).
\end{aligned} 
\] 
Simplifying,
\[ \begin{aligned}
\dfrac{d}{dt}\mathcal{H}(t) 
=&~{}- 2 \Im \int \vA  \bigg( \partial_r \overline{\phi_1} \phi_2 + \overline{\phi_2}  \partial_r  \phi_1\bigg)
-2\Im \int_0^\infty \vA'  \overline{\phi_1}\phi_2 
\\&
+2\Im \int_0^\infty \vA   \left( 
							   \frac{1}{r} \overline{\phi_1} \phi_2 
							+  \frac{S}{r}( \overline{\phi_2}\phi_1 +\overline{\phi_1} \phi_2 )
						 \right)
\\&
+2 \Im \int \vA  (\overline{\phi_2} W_2+\overline{\phi_1} W_1)
\\
=&~{}
-2\Im \int_0^\infty \left( \vA'  - \frac{1}{r}\vA \right) \overline{\phi_1}\phi_2 
+2 \Im \int \vA  (\overline{\phi_2} W_2+\overline{\phi_1} W_1)
.
\end{aligned} 
\] 
This ends the proof of the result.
\end{proof}

Of particular importante is the lack of conservation of mass in the previous computation. Therefore, we need a good choice for $\vA$. Before making this choice, we make a further estimate on the derivative of $\mathcal H$ \eqref{dH}. Since $W$ satisfies \eqref{eq:W_2D}, and Lemma \ref{lem2d} holds, we have
\[ 
\begin{aligned}
\left|\dfrac{d}{dt}\mathcal{H}(t) \right|  
\lesssim & 
\int  \left| \vA'  -\frac{1}{r}\vA   \right| |\phi|^2 +  \int \vA |\phi|^{p+1} \\
\lesssim & 
\int  \left| \vA'  -\frac{1}{r}\vA   \right| |\phi|^2 + \epsilon^2 \int \frac{\vA}{(1+r)} |\phi|^{p-1}.
\end{aligned} 
\] 
We set
\be\label{vAA}
\vA (r) = \dfrac{r}{(1+r)^3}.
\ee
Using the expression for $\vA $, we have $\vA '-\dfrac{\vA }{r} =-\dfrac{3 r}{(r+1)^4}$. Combining this with the fact that $W$ satisfies \eqref{eq:W_2D}, where  $p_{*}\geq3$ (see the conditions \emph{\ref{thm:-_m-_w}} and \emph{\ref{thm:-_m+_w}}; and \emph{\ref{thm:+_massless_Strong}} and \emph{\ref{thm:+_massive}} in Theorem \ref{thm:2D_m-} and Theorem \ref{thm:2D_m+}, respectively) and by Lemma \ref{lem2d}, we obtain
\[ 
\begin{aligned}
\left|\dfrac{d}{dt}\mathcal{H}(t) \right|  
\lesssim & 
\int  \left| \vA'  -\vA  \frac{1}{r}\right| |\phi|^2 + \epsilon^2 \int  \dfrac{\vA}{(1+r)}  |\phi|^2
\\
\lesssim & 
\int \frac{3 r}{(r+1)^4}  |\phi|^2 + \epsilon^2 \int  \frac{r}{(r+1)^4}  |\phi|^2
\lesssim 
\int \frac{r}{(r+1)^4}  |\phi|^2.
\end{aligned} 
\]
Recall that \eqref{final_K} implies that there exists a sequence $t_n \to \infty$ such that $\mathcal{H}(t_n) \to 0$. \new{More precisely, integrating \eqref{final_K} in time and using that $\K$ is uniformly bounded (see the bound on the virial functionals in Subsection~\ref{sec:thm+}), one obtains 
\be\label{integra}
\int_0^\infty\!\left(\int \tfrac{r^{2+\delta}}{(1+r)^2}|\nabla\phi|^2+\tfrac{r^{\delta}}{(1+r)^2}|\phi|^2\right)\,dt<\infty;
\ee
in particular the integrand is summable in time ($L^2_t$), so it must vanish along a sequence $t_n\to\infty$, which yields $\mathcal H(t_n)\to0$.} Indeed, we treat the case of Theorem \ref{thm:2D_m-}: from \eqref{mathH} and \eqref{vAA}, and $\delta \in (0,1)$,
\[
\ba
\mathcal{H}(t) = &~{} \int_0^1 \frac{r}{(1+r)^3}|\phi|^2 +\int_1^\infty \frac{r}{(1+r)^3}|\phi|^2 \\
\lesssim & \int_0^1 \frac{r^\delta}{(1+r)^2}|\phi|^2 + \int_1^\infty \frac{r^\delta}{(1+r)^2}|\phi|^2  = \int \frac{r^\delta}{(1+r)^2}|\phi|^2.
\ea
\]
Integrating the inequality above on $[t,t_n]$ we see that 
\[ 
\begin{aligned}
|\mathcal{H}(t_n)-\mathcal{H}(t)| =&~{} \left|\int_t^{t_n}\dfrac{d}{dt}\mathcal{H}(s)ds\right| \\
\leq &~{} \int_t^{t_n}\left|\dfrac{d}{dt}\mathcal{H}(s)\right|ds \lesssim  \int_t^{t_n} \int \frac{r}{(r+1)^4}  |\phi|^2 (s,r)dr ds ,
\end{aligned} 
\]
and passing to limit as $n \to \infty$ we have for $\delta\in (0,1)$,
\[
\mathcal{H}(t) \leq \int_t^\infty  \int \frac{r}{(r+1)^4}  |\phi|^2 (s,r)dr ds \lesssim \int_t^\infty  \int \frac{r^\delta}{(r+1)^2}  |\phi|^2 (s,r)dr ds,
\]
and hence $\displaystyle\lim_{t \to \infty}\mathcal{H}(t) =0$. 
To conclude the proof is enough to note that for any $R>0$ we have  
\[ \begin{aligned}
\|\phi\|_{ L^2(B(0,R))}^2 
 \lesssim &~{} (1+R)^3\int_0^R \dfrac{r}{(1+r)^3} |\phi|^2 dr \\
\lesssim & ~{} (1+R)^3 ~{} \mathcal{H}(t),
\end{aligned} \]
and the result \eqref{limit2D} follows.  The case of Theorem \ref{thm:2D_m+} is completely similar using \eqref{final_K+} as key integrated virial term.

\appendix

\section{Proof Claim \ref{Cl1}}\label{app:Cl1}

Let $K\in\mathbb R$, and let $\vA $ be a smooth bounded function such that $\vA (0)=0$ and let $f\in H^1$. It holds that
\[
\begin{aligned}
& \int \left[ \vA  \partial_r f +\frac12 \vA ' f \right]   \left(\partial_r+\frac{K+1}{r} \right)\left(\partial_r -\frac{K}{r}\right)f \\
&\quad = \int \left[ \vA  \partial_r f +\frac12 \vA ' f \right]   \left( \partial_r^2 f+\frac{K+1}{r} \partial_r f-\frac{K(K+1)}{r^2}f -\partial_r \left(\frac{K}{r}f\right)\right)\\
&\quad = \int \left[ \vA  \partial_r f +\frac12 \vA ' f \right]   \left( \partial_r^2 f+\frac{1}{r} \partial_r f-\frac{K^2}{r^2}f 
 \right).
\end{aligned}
\]
Expanding and rewriting, we obtain
\[
\begin{aligned}
& \int \left[ \vA  \partial_r f +\frac12 \vA ' f \right]   \left(\partial_r+\frac{K+1}{r} \right)\left(\partial_r -\frac{K}{r}\right)f \\
 &\quad = \int \vA  \left(\frac12 \partial_r[ ( \partial_r f)^2]+\frac{1}{r} (\partial_r f)^2-\frac{K^2}{2r^2} \partial_r (f^2 )
 \right)
 \\&\qquad +
 \frac12 \int  \vA '    \left( f \partial_r^2 f+\frac{1}{2r} \partial_r (f^2)-\frac{K^2}{r^2}f^2  \right)
 .
\end{aligned}
\]
Now, integrating by parts,  we have
\[
\begin{aligned}
& \int \left[ \vA  \partial_r f +\frac12 \vA ' f \right]   \left(\partial_r+\frac{K+1}{r} \right)\left(\partial_r -\frac{K}{r}\right)f \\
  &\quad =-\frac12 \int \left(2\vA' -\frac{2}{r} \vA \right)(\partial_r f)^2
 \\&\qquad 
 - \frac12 \int \bigg( 
 			\partial_r\left( \vA '  \frac{1}{2r}\right) 
 			 - \partial_r \left(\vA \frac{K^2}{r^2}\right)
 			+   \vA '  \frac{K^2}{r^2}
			- \frac12   \vA'''  
		\bigg) f ^2
		.
\end{aligned}
\]
Expanding and rearranging the terms, we conclude
\[
\begin{aligned}
& \int \left[ \vA  \partial_r f +\frac12 \vA ' f \right]   \left(\partial_r+\frac{K+1}{r} \right)\left(\partial_r -\frac{K}{r}\right)f \\
  &\quad =-\frac12 \int \left(2\vA' -\frac{2}{r} \vA \right)(\partial_r f)^2
 \\&\qquad 
 - \frac12 \int \bigg( 
 			 \vA''  \frac{1}{2r}
			 -\vA '  \frac{1}{2r^2}
 			 - \vA' \frac{K^2}{r^2}
			 +2\vA \frac{K^2}{r^3}
 			+   \vA '  \frac{K^2}{r^2}
			- \frac12   \vA'''  
		\bigg) f ^2
\\
  &\quad =-\frac12 \int \left(2\vA' -\frac{2}{r} \vA \right)(\partial_r f)^2
 \\&\qquad 
 - \frac12 \int \bigg( 
 			2\vA \frac{K^2}{r^3}
 			+ \vA''  \frac{1}{2r}
			 -\vA '  \frac{1}{2r^2}
			- \frac12   \vA'''  
		\bigg) f ^2.
\end{aligned}
\]
For the second identity, we follow the same strategy. Then, we obtain
\[
\begin{aligned}
 \int \left[ \vA  \partial_r f +\frac12 \vA ' f \right]  & \left(\partial_r-\frac{K}{r} \right)\left(\partial_r +\frac{K+1}{r}\right)f \\
&\quad =  \int \vA     \left( \frac12 \partial_r[( \partial_r f)^2] +\frac{1}{r} (\partial_r f )^2 -\frac{(K+1)^2}{2r^2}\partial_r (f^2) \right)\\
&\qquad +  \frac12\int  \vA '    \left( f \partial_r^2 f+\frac{1}{2r} \partial_r (f^2) -\frac{(K+1)^2}{r^2}f^2 \right)\\
&\quad = -\frac12  \int    \left( 2\vA' -\vA \frac{2}{r} \right)(\partial_r f )^2
\\&\qquad \quad 
		-\frac12 \int    \bigg(  2\vA \frac{(K+1)^2}{r^3} 
		-  \frac12  \vA '''  
		+ \vA ''   \frac{1}{2r} 
		-\vA '   \frac{1}{2r^2}  
		 \bigg)  f^2.
\end{aligned}
\]

\section{Proof of Proposition \ref{prop:K1} }\label{app:Proof_virial}

First, we will recall a simple but useful fact: Let $f\in H^1 $ and $\vA$ a smooth bounded continuous function with $\vA(0)=0$. Then
 \begin{equation}\label{eq:int0}\int \left( \vA  f_{r}+\frac12 \vA ' f\right)f=0.\end{equation}

\begin{proof}[Proof Proposition \ref{prop:K1}]
Let us focus on $\K_1$. Differentiating on \eqref{eq:K1} with respect of time and using \eqref{eq:SD2}, one obtain
\[
\begin{aligned}
\dt& \K_{1}\\
=&~{}\int \left[ \vA  \partial_t \partial_r \phi_{11} +\frac12 \vA ' \partial_t \phi_{11} \right]\left(  \left( \partial_r +\frac{S+1}{r}\right)  \phi_{22}-m\phi_{12} \right)
\\
&+\int \left[ \vA  \partial_r \phi_{11} +\frac12 \vA ' \phi_{11} \right]
			\left( \left( \partial_r +\frac{S+1}{r}\right) \partial_t  \phi_{22}
					-m\partial_t \phi_{12}  \right)
\\
=&~{}\int \left[ \vA   \partial_r \left(    \left( \partial_r +\frac{S+1}{r}\right)  \phi_{22}-m\phi_{12}   +W_{12} \right) 
\right. \\
&~{} \qquad \left. +\frac12 \vA ' \left(   \left( \partial_r +\frac{S+1}{r}\right)  \phi_{22}-m\phi_{12}   +W_{12} \right) \right]\\
& \qquad \times \left(  \left( \partial_r +\frac{S+1}{r}\right)  \phi_{22}-m\phi_{12} \right)
\\
&~{} +\int \left[ \vA  \partial_r \phi_{11} +\frac12 \vA ' \phi_{11} \right]
			\bigg( 
				\left( \partial_r +\frac{S+1}{r}\right)  \left(\left( \partial_r -\frac{S}{r}\right) \phi_{11}-m\phi_{21} -W_{21}\right)
				\\&\qquad \qquad\qquad
					\qquad \qquad \qquad +m\left( 
						 \left( \partial_r +\frac{S+1}{r}\right) \phi_{21}-m\phi_{11}+W_{11}
						\right)		
			 \bigg)
.
\end{aligned}
\]
Having in mind \eqref{eq:int0}, we get
\[
\begin{aligned}
\dt \K_{1}
=&~{}
\int \left[ \vA   \partial_r W_{12} 
+\frac12 \vA ' W_{12}  \right]\left(  \left( \partial_r +\frac{S+1}{r}\right)  \phi_{22}-m\phi_{12} \right)
\\
&\quad +\int \left[ \vA  \partial_r \phi_{11} +\frac12 \vA ' \phi_{11} \right]
			\\
			& \qquad \quad \times \bigg( 
				\left( \partial_r +\frac{S+1}{r}\right)  \left(\left( \partial_r -\frac{S}{r}\right) \phi_{11}-W_{21}\right)
					+m W_{11}		
			 \bigg)
.
\end{aligned}
\]
Applying Claim \ref{Cl1}, we conclude
\[
\begin{aligned}
\dt \K_{1}
=&~{}
\int \left[ \vA   \partial_r W_{12} +\frac12 \vA ' W_{12}  \right]\left(  \left( \partial_r +\frac{S+1}{r}\right)  \phi_{22}-m\phi_{12} \right)
\\ &~{}
+\int \left[ \vA  \partial_r \phi_{11} +\frac12 \vA ' \phi_{11} \right]
			\bigg( 
				- \left( \partial_r +\frac{S+1}{r}\right)  \left(W_{21}\right)
					+m W_{11}		
			 \bigg)
\\
&~{} 
-\frac12 \int \left(2\vA' -\frac{2}{r} \vA \right)(\partial_r \phi_{11}	)^2
 \\ &~{}
 - \frac12 \int \bigg(
 			2\vA \frac{\new{S}^2}{r^3}
 			+ \vA''  \frac{1}{2r}
			 -\vA '  \frac{1}{2r^2}
			- \frac12   \vA'''
		\bigg) \phi_{11}^2.
\end{aligned}
\]
Now, let us focus on $\widetilde{\K_1}$. Similar as before, taking derivative in \eqref{eq:tK1} with respect of time and using \eqref{eq:SD2}, we get
\[
\begin{aligned}
\dt \wK_{1}
=&\int \left[ \vA  \partial_r \partial_t \phi_{22} +\frac12 \vA ' \partial_t \phi_{22} \right] \left( \left( \partial_r -\frac{S}{r}\right) \phi_{11}-m\phi_{21}  \right)
\\ &+ \int \left[ \vA  \partial_r \phi_{22} +\frac12 \vA ' \phi_{22} \right]
\left( \left( \partial_r -\frac{S}{r}\right) \partial_t \phi_{11}-m\partial_t \phi_{21}  \right)
\\
=&\int \left[ \vA  \partial_r \left( \left( \partial_r -\frac{S}{r}\right) \phi_{11}-m\phi_{21} -W_{21} \right) +\frac12 \vA ' \left(  \left( \partial_r -\frac{S}{r}\right) \phi_{11}-m\phi_{21} -W_{21} \right)\right]\\
&\quad \times \left( \left( \partial_r -\frac{S}{r}\right) \phi_{11}-m\phi_{21}  \right)
\\
& + \int \left[ \vA  \partial_r \phi_{22} +\frac12 \vA ' \phi_{22} \right] \left( \partial_r -\frac{S}{r}\right) \left( \left( \partial_r +\frac{S+1}{r}\right)  \phi_{22}-m\phi_{12}   +W_{12} \right)\\
 & -m \int \left[ \vA  \partial_r \phi_{22} +\frac12 \vA ' \phi_{22} \right]    \left( \partial_r -\frac{S}{r}\right)  \left( - \left( \partial_r -\frac{S}{r}\right)  \phi_{12}+m\phi_{22}+W_{22} 
 \right).
\end{aligned}
\]
Using that \eqref{eq:int0} and Claim \ref{Cl1}, we obtain
\[
\begin{aligned}
\dt \wK_{1}
=&
-\int \left[ \vA  \partial_r \left( W_{21} \right) +\frac12 \vA ' \left(  W_{21} \right)\right]\left( \left( \partial_r -\frac{S}{r}\right) \phi_{11}-m\phi_{21}  \right)
\\& + \int \left[ \vA  \partial_r \phi_{22} +\frac12 \vA ' \phi_{22} \right] \\ 
&~{} \quad \left( 
				 \left( \partial_r -\frac{S}{r}\right) \left( \left( \partial_r +\frac{S+1}{r}\right)  \phi_{22}  +W_{12} \right)
				-m  W_{22} 
				 \right)
\\
=&
-\int \left[ \vA  \partial_r \left( W_{21} \right) +\frac12 \vA ' \left(  W_{21} \right)\right]\left( \left( \partial_r -\frac{S}{r}\right) \phi_{11}-m\phi_{21}  \right)
\\& + \int \left[ \vA  \partial_r \phi_{22} +\frac12 \vA ' \phi_{22} \right]  \left( 
				 \left( \partial_r -\frac{S}{r}\right) \left( W_{12} \right)
				-m  W_{22} 
				 \right)
\\& 
 -\frac12  \int    \left( 2\vA' -\vA \frac{2}{r} \right)(\partial_r  \phi_{22}   )^2
\\&		-\frac12 \int    \bigg(  2\vA \frac{(S+1)^2}{r^3} 
		-  \frac12  \vA '''  
		+ \vA ''   \frac{1}{2r} 
		-\vA '   \frac{1}{2r^2}  
		 \bigg)   \phi_{22}  ^2.
\end{aligned}
\]
Taking derivative with respect of time and replacing \eqref{eq:SD2}, we have
\[
\begin{aligned}
\dt& \K_{2}
\\
=&\int \left[ \vA  \partial_r \partial_t \phi_{12}+\frac12 \vA ' \partial_t \phi_{12}\right] \left(   \left( \partial_r +\frac{S+1}{r}\right)\phi_{21}-m\phi_{11}\right)
\\&
+\int \left[ \vA  \partial_r \phi_{12}+\frac12 \vA '  \phi_{12}\right] \left(   \left( \partial_r +\frac{S+1}{r}\right) \partial_t \phi_{21}-m\partial_t \phi_{11}\right)
\\
=&-\int \left[ \vA  \partial_r \left(  \left( \partial_r +\frac{S+1}{r}\right) \phi_{21}-m\phi_{11}+W_{11}\right)
\right.
\\
& \left. \qquad +\frac12 \vA ' \left(  \left( \partial_r +\frac{S+1}{r}\right) \phi_{21}-m\phi_{11}+W_{11}\right)\right] 
 \times \left(   \left( \partial_r +\frac{S+1}{r}\right)\phi_{21}-m\phi_{11}\right)
\\&
-\int \left[ \vA  \partial_r \phi_{12}+\frac12 \vA '  \phi_{12}\right]
\\&\qquad\times  \left[   
			\left( \partial_r +\frac{S+1}{r}\right) \left(  \left( \partial_r -\frac{S}{r}\right)  \phi_{12}-m\phi_{22}-W_{22}\right)
			\right.
			\\
			&\qquad \qquad \left. +m\left(  \left( \partial_r +\frac{S+1}{r}\right)  \phi_{22}-m\phi_{12}   +W_{12}\right)
			\right]
			.
\end{aligned}
\]
Now, using \eqref{eq:int0} along with \eqref{Cl1},  we obtain a simplified version of the above identity

\[
\begin{aligned}
\dt \K_{2}
=&- \int \left[ \vA  \partial_r \left( W_{11}\right)
+\frac12 \vA ' W_{11} \right] \left(   \left( \partial_r +\frac{S+1}{r}\right)\phi_{21}-m\phi_{11}\right)
\\&
-\int \left[ \vA  \partial_r \phi_{12}+\frac12 \vA '  \phi_{12}\right]
  \left(   
			\left( \partial_r +\frac{S+1}{r}\right) \left(-W_{22}\right)
			+m W_{12}
			\right)
\\&
+\frac12 \int \left(2\vA' -\frac{2}{r} \vA \right)(\partial_r  \phi_{12})^2
 \\&\qquad
 + \frac12 \int \bigg(
 			2\vA \frac{\new{S}^2}{r^3}
 			+ \vA''  \frac{1}{2r}
			 -\vA '  \frac{1}{2r^2}
			- \frac12   \vA'''
		\bigg)  \phi_{12} ^2
			.
\end{aligned}
\]
Following a similar procedure for the last identity, we obtain
\[\begin{aligned}
\dt& \widetilde{\K}_{2}
\\
=&\int \left[ \vA  \partial_r \partial_t \phi_{21}+\frac12 \vA '  \partial_t  \phi_{21}\right]\left( \left( \partial_r -\frac{S}{r}\right)  \phi_{12}-m\phi_{22}\right)
\\&
+\int \left[ \vA  \partial_r \phi_{21}+\frac12 \vA '   \phi_{21}\right]\left( \left( \partial_r -\frac{S}{r}\right) \partial_t   \phi_{12}-m \partial_t \phi_{22}\right)
\\
=&-\int \left[ \vA  \partial_r  \left(  \left( \partial_r -\frac{S}{r}\right)  \phi_{12}-m\phi_{22}-W_{22} \right)+\frac12 \vA '  \left(  \left( \partial_r -\frac{S}{r}\right)  \phi_{12}-m\phi_{22}-W_{22} \right) \right]
\\&\qquad \times \left( \left( \partial_r -\frac{S}{r}\right)  \phi_{12}-m\phi_{22}\right)
\\&
-\int \left[ \vA  \partial_r \phi_{21}+\frac12 \vA '   \phi_{21}\right]
\\&\qquad \times \left[
		 \left( \partial_r -\frac{S}{r}\right)\left( \left( \partial_r +\frac{S+1}{r}\right) \phi_{21}-m\phi_{11}+W_{11} \right)
		 \right.
		 \\&\qquad\qquad 
		 \left. +m \left(  \left( \partial_r -\frac{S}{r}\right) \phi_{11}-m\phi_{21} -W_{21} \right)
		 \right]
		 .
\end{aligned}
\]
Rearranging the terms and having in mind \eqref{eq:int0}, we obtain 
\[\begin{aligned}
\dt \widetilde{\K}_{2}
=&\int \left[ \vA  \partial_r  \left( W_{22} \right)+\frac12 \vA '  W_{22}  \right]
 \left( \left( \partial_r -\frac{S}{r}\right)  \phi_{12}-m\phi_{22}\right)
\\&
-\int \left[ \vA  \partial_r \phi_{21}+\frac12 \vA '   \phi_{21}\right]
\left[
		 \left( \partial_r -\frac{S}{r}\right)\left( \left( \partial_r +\frac{S+1}{r}\right) \phi_{21}+W_{11} \right)
		 -m W_{21} 
		 \right]
\\
=&\int \left[ \vA  \partial_r  \left( W_{22} \right)+\frac12 \vA '  W_{22}  \right]
 \left( \left( \partial_r -\frac{S}{r}\right)  \phi_{12}-m\phi_{22}\right)
\\&
-\int \left[ \vA  \partial_r \phi_{21}+\frac12 \vA '   \phi_{21}\right]
\left(
		 \left( \partial_r -\frac{S}{r}\right)\left( W_{11} \right)-m W_{21} 
		 \right)
\\&
 +\frac12  \int    \left( 2\vA' -\vA \frac{2}{r} \right)(\partial_r  \phi_{21} )^2
\\
& +\frac12 \int    \bigg(  2\vA \frac{(\new{S}+1)^2}{r^3}
				-  \frac12  \vA '''
				+ \vA ''   \frac{1}{2r}
				-\vA '   \frac{1}{2r^2}
			 \bigg)   \phi_{21}^2
			 .
\end{aligned}
\]

This ends the proof of the identity.
\end{proof}

\subsection*{Data Availability} All the data obtained for this work is presented in the same manuscript.

\subsection*{Conflict of Interest} The authors declare no conflict of interest in the production and possible publication of this work.


\begin{thebibliography}{99}

\bibitem{AY}   M.J. Ablowitz and Y. Zhu, \emph{Nonlinear wave packets in deformed honeycomb lattices}, SIAM J. Appl. Math. 73 (2013), 1959--1979.

\bibitem{AC_NLS_NonExistence} M. \'A. Alejo, A. J. Corcho, \emph{On the nonexistence of NLS breathers}, Phys. D 475 (2025), Paper No. 134580, 11 pp.

\bibitem{ACKM} M. \'A. Alejo, F. Cortez, C. Kwak and C. Mu\~noz, \emph{On the dynamics of zero-speed solutions for Camassa-Holm type equations}, IMRN, rnz038, \url{https://doi.org/10.1093/imrn/rnz038}.

\bibitem{AleMau}  M. \'A. Alejo, and C. Maul\'en, \emph{Decay for Skyrme wave maps}, Lett. Math. Phys. 112 (2022), no. 5, Paper No. 90, 33 pp.

\bibitem{BH} I. Bejenaru and S. Herr, \emph{The cubic Dirac equation: Small initial data in $H^{1/2}(\R^2)$}, Commun. Math. Phys. 343, 515--562 (2016)

\bibitem{B_SS_Kerr} W. Borrelli, \emph{Stationary solutions for the 2D critical Dirac equation with Kerr nonlinearity}. J. Differ. Equ. 263, 7941--7964 (2017)

\bibitem{B_loc} W. Borrelli, \emph{Weakly localized states for nonlinear Dirac equations}, Calc. Var. (2018) 57:155.

\bibitem{B_Sym_2d} W. Borrelli,  \emph{Symmetric solutions for a 2d critical dirac equation}, Commun. Contemp. Math. 24, 2150019 (2022). \url{https://doi.org/10.1142/S021919972150019X}

\bibitem{BF_decay} W.  Borrelli and R.L.  Frank, \emph{Sharp decay estimates for critical Dirac equations}. Trans. Am. Math. Soc. 373, 2045--2070 (2020).

\bibitem{bubbles} W. Borrelli, A. Malchiodi, \& R. Wu, \emph{ Ground state Dirac bubbles and Killing spinors}. Commun. Math. Phys. 383, 1151--1180 (2021). \href{https://doi.org/10.1007/s00220-021-04013-1}.

\bibitem{BC_2d} N. Bournaveas and T. Candy. \emph{Global Well-Posedness for the Massless Cubic Dirac Equation}. International Mathematics Research Notices 2016, n. 22, 6735-6828. \url{https://doi.org/10.1093/imrn/rnv361}.

\bibitem{BDAF} N. Boussa\"id, P. D'Ancona,  and L. Fanelli, \emph{Virial identity and weak dispersion for the magnetic Dirac equation}, J. Math. Pures Appl. 95 (2011) 137--150. 

\bibitem{Cacciafesta} F. Cacciafesta, \emph{Virial Identity and Dispersive estimates for the n-Dimensional Dirac Equation}, J. Math. Sci. Univ. Tokyo 18 (2011), 441--463.

\bibitem{CF_} F. Cacciafesta, L. Fanelli,  \emph{ Dispersive estimates for the Dirac equation in an Aharonov-Bohm field}. Journal of Differential Equations 263, 4382--4399 (2017). \url{https://doi.org/10.1016/j.jde.2017.05.018}

\bibitem{CH23}  T. Candy and S. Herr, \emph{The massless and the non-relativistic limit for the cubic Dirac equation}, preprint  \href{https://arxiv.org/abs/2308.12057}{arXiv:2308.12057} (2023).

\bibitem{CS_2016} F. Cacciafesta, E. S\'er\'e,  \emph{Local smoothing estimates for the massless Dirac--Coulomb equation in 2 and 3 dimensions}. Journal of Functional Analysis 271, 2339--2358. \url{https://doi.org/10.1016/j.jfa.2016.04.003}
\new{\bibitem{CazenaveVazquez} T. Cazenave and L. V\'azquez, \emph{Existence of localized solutions for a classical nonlinear Dirac field}, Comm. Math. Phys. 105 (1986), no. 1, 35--47.}
\bibitem{CKSCL} J. Cuevas-Maraver, P.G. Kevrekidis,  A. Saxena,  A. Comech, and R. Lan,  \emph{Stability of Solitary Waves and Vortices in a 2D Nonlinear Dirac Model}, Phys. Rev. Lett. 116, 214101.

\bibitem{CBCKS_2018} J. Cuevas-Maraver, N. Boussa\"id, A. Comech, R. Lan, P.G. Kevrekidis, and A. Saxena (2018). \emph{Solitary Waves in the Nonlinear Dirac Equation}. In: Carmona, V., Cuevas-Maraver, J., Fern\'andez-S\'anchez, F., Garc\'ia- Medina, E. (eds) Nonlinear Systems, Vol. 1. Understanding Complex Systems. Springer, Cham. \url{https://doi.org/10.1007/978-3-319-66766-9_4}.

\bibitem{DAFS}  P. D'Ancona, and L. Fanelli, and N. M. Schiavone, \emph{Eigenvalue bounds for non-selfadjoint Dirac operators}. Math. Ann. 383 (2022), no. 1-2, 621--644.

\bibitem{Dirac}  P. A. M. Dirac, The Principles of Quantum Mechanics, Oxford University Press, USA, 1982.

\bibitem{EG} M.B. Erdo\u{g}an, W.R Green, \emph{The Dirac Equation in Two Dimensions: Dispersive Estimates and Classification of Threshold Obstructions}. Commun. Math. Phys. 352, 719--757 (2017). \url{https://doi.org/10.1007/s00220-016-2811-8}
\new{
\bibitem{EstebanSere} M. J. Esteban and E. S\'er\'e, \emph{Stationary states of the nonlinear Dirac equation: a variational approach}, Comm. Math. Phys. 171 (1995), no. 2, 323--350.} 

\bibitem{Honey2} C.L. Fefferman, J.P. Lee-Thorp,  and M.I. Weinstein,  \emph{ Edge States in Honeycomb Structures}. Ann. PDE 2, 12 (2016). \href{https://doi.org/10.1007/s40818-016-0015-3}.

\bibitem{FW1} C. L. Fefferman, M. I. Weinstein, \emph{Honeycomb lattice potentials and Dirac points}, J. Am. Math. Soc. 25, 1169-1220 (2012)

\bibitem{FW2} C. L. Fefferman, M. I. Weinstein, \emph{Wave packets in honeycomb structures and two-dimensional Dirac equations}, Comm. Math. Phys. 326, 251-286 (2014)

\bibitem{Honey} C. L. Fefferman, and M. I. Weinstein, \emph{Waves in honeycomb structures}, Journ\'ees \'Equations aux deriv\'ees partialles, Biarritz, 3-7 juin 2012, GDR 243 4 (CNRS).

\bibitem{GS_hartree} V. Georgiev, and B. Shakarov, \emph{Global Large Data Solutions for 2D Dirac Equation with Hartree Type Interaction}, IMRN, Volume 2022, Issue 17, August 2022, Pages 12803--12820, \url{https://doi.org/10.1093/imrn/rnab082}.

\bibitem{HMM} S. Herr, C. Maul\'en, and C. Mu\~noz, \emph{Decay of solutions of nonlinear Dirac equations}, {\color{blue} Comm. Math. Phys. 407 (2026), no. 5, Paper No. 94, 52 pp.} 

\bibitem{KMM_Nonexistence_KG} M. Kowalczyk, Y. Martel, Y. \& C. Mu\~noz, \emph{Nonexistence of small, odd breathers for a class of nonlinear wave equations}. Lett. Math. Phys. 107, 921--931 (2017). \url{https://doi.org/10.1007/s11005-016-0930-y}.

\bibitem{EM20} M. E. Martinez, \emph{Decay of small odd solutions for long range Schr\"odinger and Hartree equations in one dimension}. Nonlinearity 33, 1156--1182 (2020).

\bibitem{MoMu} M. Morales and C. Mu\~noz, \emph{On local decay of inflaton and axion fields}. Partial Differ. Equ. Appl. 5 (2024), no. 3, Paper No. 19, 30 pp. 

\bibitem{OY04} T. Ozawa and K. Yamauchi. \emph{Structure of Dirac matrices and invariants for nonlinear Dirac equations}. Diff. Int. Eqns. 17 (9-10) 971 - 982, 2004. \url{https://doi.org/10.57262/die/1356060310}.

\bibitem{pecher1} H. Pecher, \emph{Local well-posedness for the nonlinear Dirac equation in two space dimensions}, Communications on Pure and Applied Analysis, 2014, 13(2): 673-685. doi: 10.3934/cpaa.2014.13.673 

\bibitem{pecher2} H. Pecher, Corrigendum of \emph{Local well-posedness for the nonlinear Dirac equation in two space dimensions}, Commun. Pure Appl. Anal. 14 (2015), no. 2, 737-742.

\bibitem{Transverse_Pelinovsky} D. Pelinovsky, Y. Shimabukuro, \emph{ Transverse Instability of Line Solitary Waves in Massive Dirac Equations}. J Nonlinear Sci 26, 365--403 (2016). \href{https://doi.org/10.1007/s00332-015-9278-1}.

\bibitem{soler} M. Soler, \emph{Classical, stable, nonlinear, spinor field with positive rest energy}, Phys. Rev. D {\bf 1} (1970), pp. 2766--2769.

\bibitem{Thaller} B. Thaller, \emph{The Dirac Equation}, Springer-Verlag Berlin Heidelberg 1992.

\bibitem{T58} W. E. Thirring, \emph{A soluble relativistic field theory}, Ann. Phys., 3 (1958), pp. 91--112.

\bibitem{wakano} M. Wakano, \emph{ Intensely Localized Solutions of the Classical Dirac-Maxwell Field Equations}, Progress of Theoretical Physics, Vol. 35, Issue 6, June 1966, Pages 1117--1141, \url{https://doi.org/10.1143/PTP.35.1117}.

\end{thebibliography}
\end{document}